\documentclass[11pt, a4paper, english]{amsart}
\usepackage{amsmath} 
\usepackage{amsthm} 
\usepackage{amssymb} 
\usepackage[dvipsnames]{xcolor}
\usepackage{amscd} 
\usepackage{mathtools}
\usepackage{booktabs}
\usepackage[boxed]{algorithm2e}
\usepackage{subcaption}
\usepackage{array} 
\usepackage[cal=boondoxupr,scr=euler]{mathalfa}
\usepackage[backref=page,linktocpage]{hyperref} 
\usepackage{caption} %
\usepackage{graphics,graphicx} 
\usepackage{tikz,tikz-cd} 
\usetikzlibrary{graphs,graphs.standard,calc}
\usetikzlibrary{decorations.pathreplacing,}
\usepackage[shortlabels]{enumitem} 

\DeclareMathAlphabet{\mathsf}{OT1}{\sfdefault}{m}{n}

\newcommand{\nocontentsline}[3]{}
\newcommand{\tocless}[2]{\bgroup\let\addcontentsline=\nocontentsline#1{#2}\egroup}

\usepackage[margin=1.30in]{geometry} 

\usepackage{verbatim}

\usepackage{scalerel}

\makeatletter
\def\dual#1{\expandafter\dual@aux#1\@nil}
\def\dual@aux#1/#2\@nil{\begin{tabular}{@{}c@{}}#1\\#2\end{tabular}}
\makeatother

\makeatletter
\@namedef{subjclassname@2020}{\textup{2020} Mathematics Subject Classification}
\makeatother

\DeclareMathAlphabet{\amathbb}{U}{bbold}{m}{n}

\hypersetup{
    colorlinks = true,
    linkbordercolor = {white},
    linkcolor = {NavyBlue},
    anchorcolor = {black},
    citecolor = {NavyBlue},
    filecolor = {cyan},
    menucolor = {NavyBlue},
    runcolor = {cyan},
    urlcolor = {NavyBlue}
}

\usetikzlibrary{automata}

\newtheoremstyle{teoremas}
{12pt}
{13pt}
{\itshape}
{}
{\bfseries}
{}
{.5em}
{}

\theoremstyle{teoremas}
\newtheorem{theorem}{Theorem}[section]
\newtheorem{corollary}[theorem]{Corollary}
\newtheorem{lemma}[theorem]{Lemma}
\newtheorem{proposition}[theorem]{Proposition}

\newtheoremstyle{definition}
{12pt}
{12pt}
{}
{}
{\bfseries}
{}
{.5em}
{}

\theoremstyle{definition}
\newtheorem{definition}[theorem]{Definition}
\newtheorem{conjecture}[theorem]{Conjecture}

\newtheorem{problem}[theorem]{Problem}
\newtheorem{question}[theorem]{Question}
\newtheorem{example}[theorem]{Example}
\newtheorem{remark}[theorem]{Remark}

\newcommand{\M}{\mathsf{M}}
\newcommand{\N}{\mathsf{N}}
\newcommand{\U}{\mathsf{U}}

\DeclareMathOperator{\rk}{rk}

\newcommand{\girth}{\operatorname{girth}}

\DeclareMathOperator*{\conv}{conv}

\newcommand{\indicator}[1]{\overline{\amathbb{1}}\!\left({#1}\right)}
\newcommand{\symindicator}[1]{\amathbb{1}\!\left({#1}\right)}

\AtBeginDocument{%
   \def\MR#1{}
}

 \usepackage[colorinlistoftodos,bordercolor=orange,backgroundcolor=orange!20,linecolor=orange,textsize=scriptsize]{todonotes}
\title{The polytope of all matroids}

\author{Luis Ferroni\and Alex Fink}
\address{(L. Ferroni)
 Universit\`a di Pisa, Pisa, Italy 
}
\email{luis.ferroni@unipi.it}
\address{(A. Fink)
 Queen Mary University of London, London, United Kingdom
}
\email{a.fink@qmul.ac.uk}


\allowdisplaybreaks

\begin{document}

\begin{abstract}
    It is possible to write the indicator function of any matroid polytope as an integer combination of indicator functions of Schubert matroid polytopes. In this way, every matroid on $n$ elements of rank~$r$ can be thought of as a lattice point in the space having a coordinate for each Schubert matroid on $n$ elements of rank~$r$. 
    We study the convex hull of all these lattice points, with particular focus on the vertices, which come from the matroids we call extremal matroids. We show that several famous classes of matroids arise as faces of the polytopes, and in many cases we determine the dimension of this face explicitly. As an application, we show that there exist valuative invariants that attain non-negative values at all representable matroids, but fail to be non-negative in general.  
\end{abstract}

\maketitle

\allowdisplaybreaks

\section{Introduction}\label{sec:one}

Many open problems in matroid theory, especially within combinatorial Hodge theory, can be stated as conjectural non-negativity of matroid invariants. 
Often, these invariants are \emph{valuative} under matroid polytope subdivisions: a valuative function provides a notion of ``measure'' for matroid polytopes.
This property motivates the use of the tools of polyhedral geometry even if the source of the original problem is removed from the point of view of polytopes.  

Given a matroid $\M$ with ground set $[n]$ of rank $r$, a result of Derksen and Fink \cite{derksen-fink} 
allows us to valuatively decompose the indicator function of the base polytope $\mathscr{P}(\M)$ as an integer linear combination of indicator functions of base polytopes of Schubert matroids 
(a.k.a.\ nested matroids) of the same ground set and rank. Moreover, this decomposition is always unique. Thus, to each matroid one can associate its vector of coordinates in the basis of Schubert matroids. We provide a careful recapitulation of these notions in Section~\ref{sec:preliminaries}. 

The main subjects of this article are two polytopes,
not themselves matroid polytopes but rather \emph{parameter polytopes} for these Schubert expansions. 
The first is the polytope $\overline{\Omega}_{r,n}$ obtained by taking the convex hull of all points associated to matroids on a fixed ground set $[n]$ of rank $r$, using Schubert expansions to provide coordinates. 
The second is its unlabelled counterpart, the polytope $\Omega_{r,n}$ obtained by taking the convex hull of all isomorphism classes of matroids using the coordinates given by the isomorphism classes of Schubert matroids. We introduce these polytopes in Section~\ref{sec:three}.

In the polytope $\overline{\Omega}_{r,n}$ every matroid is located at a vertex, but the situation for $\Omega_{r,n}$ is subtler. This leads to a new notion that we study in Section~\ref{sec:four}: that of \emph{extremal matroid}. The definition of $\Omega_{r,n}$ implies that extremal matroids are precisely those matroids for which some valuative invariant attains its global maximum. This makes them natural candidates to address the validity of positivity conjectures in matroid theory. One can view $\Omega_{r,n}$ as a discrete object which parametrizes potential counterexamples to conjectures like those in our first paragraph. 

Section~\ref{sec:faces} investigates the facial structure of our polytopes. We show that various special classes of matroids (e.g., disconnected matroids, simple matroids, sparse paving matroids, paving matroids, and elementary split matroids) correspond to faces of the polytopes $\Omega_{r,n}$ and $\overline{\Omega}_{r,n}$. In some cases we compute the dimension of these faces. For example, sparse paving matroids form a $1$-dimensional face, i.e.\ an edge, of $\Omega_{r,n}$. 
On account of this observation we expect the polytopes $\Omega_{r,n}$ to have ``few'' vertices, compared to the number of all isomorphism classes of matroids on $n$ elements of rank~$r$.

Section~\ref{sec:extremal-non-representable} addresses a folklore question in matroid theory that provided a large part of the inspiration for this paper. Is there a valuative invariant that takes non-negative values at all representable matroids but fails to be non-negative in general? We answer this question affirmatively. 
In the polytopal language, what we show is that $\Omega_{r,n}$ has a vertex
such that all of the extremal matroids occupying it are non-representable.

\subsection*{Acknowledgements}
    The authors thank Matt Larson for motivating discussions about non-negative valuations on representable matroids, and Joseph Bonin, Christian Haase, and Benjamin Schr\"oter for helpful and inspiring conversations about structural properties of the polytopes introduced in the present paper.
    We thank the Bernoulli Centre for Fundamental Studies at EPFL,
    where this work was conceived,
    and the Institute for Advanced Study, where it was completed, for their support.
    The first author received support from the Minerva Research Foundation, at the School of Mathematics of the IAS. The second author received support from 
    the Engineering and Physical Sciences Research Council [grant number EP/X001229/1].

\section{Preliminaries}\label{sec:preliminaries}

This paper assumes basic familiarity with
matroid theory, including matroid polytopes. 
For undefined concepts in general matroid theory, we refer to \cite{oxley,welsh}, whereas for more details about matroid polytopes we suggest \cite{ardila-fink-rincon,derksen-fink,ferroni-schroter}.

A \emph{Schubert matroid} is a matroid whose lattice of cyclic flats is a chain. 
Readers who have encountered other definitions of Schubert matroid should be aware that we do not privilege one order of the ground set.

Schubert matroids are isomorphic to lattice path matroids whose lower path coincides with the bottom-right border of the grid. We refer the reader to \cite{bonin-demier} for background about lattice path matroids. 
We follow the conventions of \cite[Section~2.2]{ferroni-schroter}, which we review:
the bases of a \emph{lattice path matroid} $\M[L,U]$ defined by a lower path $L$ and an upper path $U$ correspond to the positions of the upward steps in each of the lattice paths that lie between $L$ and $U$.
When we depict a Schubert matroid as a lattice path matroid, we highlight the position of the upper path, with the lower path taken to be the lower-right border of the grid. 
For example, in Figure~\ref{fig:schuberts-4-2} (Example~\ref{example:schuberts-4-2}), the second matroid from the left 
is
a matroid of rank $2$ on $E=\{1,2,3,4\}$ with exactly two bases: $\{2,4\}$ and $\{3,4\}$.

There is a unique isomorphism class of connected matroids on $n$ elements of rank~$r$ attaining the minimum number of bases.
This class includes the Schubert matroid $\mathsf{T}_{r,n}$ 
whose lattice path diagram has upper path $U=(0,1)(1,0)^{n-r-1}(0,1)^{r-1}(1,0)$,
i.e.\ is a hook-shaped Young diagram.
Dinolt \cite{dinolt} gave the name \emph{minimal matroids} to the matroids isomorphic to $\mathsf{T}_{r,n}$, which was popularized by the first author in \cite{ferroni-minimal}.

We will use the following notation for sets of isomorphism classes of matroids throughout this article.
    \begin{align*}
        \mathcal{M}_{r,n} &= \{\text{Matroids on $[n]$ of rank $r$ up to isomorphism}\},\\
        \mathcal{S}_{r,n} &= \{\text{Schubert matroids on $[n]$ of rank $r$ up to isomorphism}\}.
    \end{align*}
In general, when we do not wish to identify isomorphic matroids, 
we use barred symbols.
So our sets of labelled matroids are
    \begin{align*}
        \overline{\mathcal{M}}_{r,n} &= \{\text{Matroids on $[n]$ of rank $r$}\},\\
        \overline{\mathcal{S}}_{r,n} &= \{\text{Schubert matroids on $[n]$ of rank $r$}\}.
    \end{align*}

\subsection{Indicator functions and valuations}\label{sec:valuations}

In this section we review the fact that indicator functions of matroid polytopes can be written as integer combinations of indicator functions of matroid polytopes of Schubert matroids. That observation underlies the definition of the new polytopes we study below. 

Throughout this article we will denote by $\mathscr{P}(\M)$ the base polytope of $\M$, and by $\mathscr{P}_{\mathrm I}(\M)$ the independence polytope of $\M$. 

We will mildly expand upon a result of Hampe \cite[Theorem~3.12]{hampe}.
Hampe explained how to write Bergman fans of loopless matroids as linear combinations of Bergman fans of loopless Schubert matroids,
but the same relations hold among base polytopes, and the looplessness assumption is inessential.
Let us start by recapitulating key notions. 

\begin{definition}
    Let $X$ be a subset of $\mathbb{R}^n$.
    Its \emph{indicator function} $\indicator{X}:\mathbb{R}^n\to\mathbb Z$ is defined by
    \[ \indicator{X}(x) = \begin{cases} 1 & x\in X,\\ 0 & \text{otherwise}.\end{cases}\]
\end{definition}

The indicator function we will use most belongs to the matroid base polytope, $\indicator{\mathscr{P}(\M)}$. 
We will also make use of $\indicator{\mathscr{P}_{\mathrm I}(\M)}$.
As suggested by the overbar in the notation,
we will introduce a variant of the above definition adapted to isomorphism classes
in Definition~\ref{def:symindicator}.

\begin{definition}\label{def:valuation}
    Let $A$ be an abelian group. A map $f:\overline{\mathcal{M}}_{r,n}\to A$ is said to be a \emph{valuative function} if
        \[ \sum_{i=1}^s a_i f(\M_i) = 0\enspace \text{ whenever }\enspace \sum_{i=1}^s a_i \indicator{\mathscr{P}(\M_i)} \text{ is the identically zero function}\]
    for coefficients $a_i \in \mathbb{Z}$.
\end{definition}

Any relation of the form $\sum_{i=1}^s a_i \indicator{\mathscr{P}(\M_i)}=0$ is called a \emph{valuative relation}. 
One set of valuative relations that has received attention
are those encoding ``the inclusion-exclusion principle'' for \emph{matroid subdivisions}. 
A matroid subdivision of~$\mathscr{P}(\M)$ is a polyhedral complex all of whose cells are base polytopes of matroids and for which the union of all cells yields $\mathscr{P}(\M)$.
Whenever a matroid subdivision of~$\mathscr{P}(\M)$ is exhibited,
any valuative function $f$ satisfies the equality 
\begin{equation}\label{eq:inc-exc}
    f(\M) = \sum_{J\subseteq [s]} (-1)^{|J| - 1} f\big{(}\bigcap_{j \in J} \M_j\big{)}
\end{equation}
where $\mathscr{P}(\M_1),\ldots, \mathscr{P}(\M_s)$ are the maximal cells of the subdivision, and $\bigcap_{j\in J} \M_j$ stands for the matroid whose base polytope is the intersection of the base polytopes of all the $\M_j$.
If this intersection is empty then the symbol $f(\bigcap_{j\in J} \M_j)$ should be read as~0.

\begin{remark}\label{rem:definitions-of-valuation}
Several notions of ``valuation'' appear in the literature:
\cite[Appendix~A]{stellahedral} compares five of them.
Definition~\ref{def:valuation} is the strongest one.
Each of the other definitions prescribes a list of valuative relations to be satisfied arising from a certain kind of configuration of convex bodies (in our case, matroid polytopes),
and therefore is strictly weaker when the list fails to generate the group of all valuative relations.
Derksen and Fink \cite[Theorem~3.5]{derksen-fink} showed that 
the valuative relations underlying \eqref{eq:inc-exc} do generate the group,
and therefore the assumption \eqref{eq:inc-exc} on~$f$
is equivalent to Definition~\ref{def:valuation}. For a more detailed treatment on valuations on general convex polyhedra we refer to \cite[Chapter~7]{barvinok}.
\end{remark}

\begin{example}\label{ex:constants-are-valuations}
    A constant function $f:\overline{\mathcal{M}}_{r,n}\to A$ is a valuation. This fact, which is not obvious a~priori, is proved in \cite[Proposition~4.4]{ardila-fink-rincon}.
\end{example}

\begin{example}\label{ex:compendium}
    Besides constants, many well-known and useful functions are valuations. We list below a few of them that we will need in the present paper.
    \begin{itemize}
        \item For any subset $B\subseteq [n]$ of size $|B|=r$, the function $f_B:\overline{\mathcal{M}}_{r,n}\to \mathbb{Z}$ defined by
    \begin{equation} \label{eq:indicator-basis}
        f_B(\M) = \begin{cases} 1 & \text{if $B$ is a basis of $\M$,}\\  0 & \text{otherwise,}\end{cases}
    \end{equation}
    is a valuation by definition,
    because it is the evaluation of $\indicator{\mathscr{P}(\M)}$ at the zero-one indicator vector of~$B$.
    \item For a given set $S\subseteq [n]$, the map $r_S:\overline{\mathcal{M}}_{r,n}\to \mathbb{Z}$ given by
    \[ r_S(\M) = \rk_{\M}(S)\]
    is a valuation \cite[Theorem~5.1]{ardila-fink-rincon}.
    \item The number of sets of a given rank and size \cite[Theorem~5.4]{ardila-fink-rincon}. In particular, the Tutte polynomial and its linear specializations such as the beta invariant, or the number of bases of the matroid.
        \item The number of (chains of) flats of with prescribed sizes and ranks \mbox{\cite[Theorem~5.6]{bonin-kung},} \cite[Theorem~8.2]{ferroni-schroter}.
        \item The number of circuits with a prescribed size \cite[Corollary~5.3]{bonin-kung}.
    \end{itemize}
\end{example}

We will sometimes find it useful to consider relations among indicator functions of \emph{independence} polytopes. This gives exactly the same theory of matroid valuations:

\begin{lemma}\label{lemma:valuative-relations-independence-polytope}
    Let $\M_1,\ldots,\M_s \in \overline{\mathcal{M}}_{r,n}$. The following equivalence holds:
    \[ \sum_{i=1}^s a_i \indicator{\mathscr{P}(\M_i)} = 0 \enspace \iff \enspace \sum_{i=1}^s a_i \indicator{\mathscr{P}_{\mathrm I}(\M_i)} = 0.\]
\end{lemma}

For a proof we refer to \cite[Proposition~2.1]{lee-patel-spink-tseng} or \cite[Proposition~7.4]{stellahedral}. 

\subsection{Schubert matroids form a basis}
A key result is \cite[Theorem~5.4]{derksen-fink}, restated next.
A second proof appears as \cite[Corollary~7.9]{stellahedral}.

\begin{theorem}\label{thm:DF}
    For every matroid $\M\in \overline{\mathcal{M}}_{r,n}$, the indicator function $\indicator{\mathscr{P}(\M)}$ can be written uniquely as an integer combination
    \begin{equation}\label{eq:integer-combination-of-Schuberts}        
         \indicator{\mathscr{P}(\M)} = \sum_{\mathsf{S} \in \overline{\mathcal{S}}_{r,n}} a_{\mathsf{S}} \indicator{\mathscr{P}(\mathsf{S})}.
    \end{equation}
\end{theorem}

\begin{remark}
     It follows from the definitions that for any valuative function $f:\overline{\mathcal{M}}_{r,n}\to A$, the evaluation of $f$ at the matroid $\M$ can be computed as 
    \[f(\M) = \sum_{\mathsf{S} \in \overline{\mathcal{S}}_{r,n}} a_{\mathsf{S}} f(\mathsf{S}).\]
    Another immediate consequence of the above result is that valuative functions $f:\overline{\mathcal{M}}_{r,n}\to A$ correspond bijectively to set functions $f:\overline{\mathcal{S}}_{r,n}\to A$.
\end{remark}

\begin{remark}\label{remark:loops-coloops}
    If the element $i\in E$ is a loop (resp.\ coloop) of the matroid $\M$, then all the Schubert matroids $\mathsf{S}$ appearing on the right hand side of equation~\eqref{eq:integer-combination-of-Schuberts} with nonvanishing coefficient, $a_\mathsf{S}\neq 0$, must have $i$ as a loop (resp.\ coloop),
    and the equation remains true when $i$ is deleted from every matroid. 
    This is immediate from the fact that $i$ being a loop (resp.\ coloop) is equivalent to $\mathscr{P}(\M)$ being contained in the hyperplane defined by the equation $x_i=0$ (resp. $x_i=1$),
    and the uniqueness of the decomposition:
    the relation \eqref{eq:integer-combination-of-Schuberts} for $\mathscr{P}(\M\setminus i)$, inside this hyperplane,
    must provide the relation \eqref{eq:integer-combination-of-Schuberts} for~$\mathscr{P}(\M)$.
\end{remark}

The next lemma relates the Schubert decomposition of $\M$ with that of its truncation. Recall that the truncation $\operatorname{tr}(\M)$ of a matroid $\M$ is the matroid having as independent sets all the independent sets in~$\M$ of corank $\ge1$ (i.e.\ other than bases). 

\begin{lemma}\label{lemma:truncation}
    Let $\M\in \overline{\mathcal{M}}_{r,n}$ be a matroid such that
        \begin{equation}\label{eq:trunc0}
        \indicator{\mathscr{P}(\M)} = \sum_{\mathsf{S} \in \overline{\mathcal{S}}_{r,n}} a_{\mathsf{S}} \indicator{\mathscr{P}(\mathsf{S})}.
        \end{equation}
    Then, the indicator function of the truncated matroid $\operatorname{tr}(\M)\in \overline{\mathcal{M}}_{r-1,n}$ satisfies
        \begin{equation}\label{eq:trunc1}
        \indicator{\mathscr{P}(\operatorname{tr}(\M))} = \sum_{\mathsf{S} \in \overline{\mathcal{S}}_{r,n}} a_{\mathsf{S}} \indicator{\mathscr{P}(\operatorname{tr}(\mathsf{S}))}.
        \end{equation}
\end{lemma}

\begin{proof}
    If one assumes equation~\eqref{eq:trunc0}, then by Lemma~\ref{lemma:valuative-relations-independence-polytope} the same linear relation holds among independence polytopes. 
    From there, restricting all functions to 
    $\{x\in\mathbb{R}^n : \linebreak \sum_{i=1}^n x_i\le r-1\}$
    and re-extending by zero
    shows that the counterpart of equation~\eqref{eq:trunc1} holds for independence polytopes. Switching back to base polytopes via Lemma~\ref{lemma:valuative-relations-independence-polytope} completes the proof.
\end{proof}

\begin{remark}
    The truncation of a Schubert matroid is itself a Schubert matroid:
    in the lattice path picture, one cuts off the top row of boxes and adds a new column on the right.
    Hence, the right hand side of equation~\eqref{eq:trunc1} 
    is an expansion into indicator functions of Schubert matroids.  
\end{remark}

\subsection{Effective computation of the Schubert coefficients}

Effective computation of the coefficients $a_{\mathsf{S}}$ in~\eqref{eq:integer-combination-of-Schuberts} is somewhat cumbersome. 
However, following an approach by Hampe \cite{hampe} and Ferroni \cite[Section~5]{ferroni}, it is possible to describe these coefficients combinatorially.

\begin{definition}
    Let $P$ be a finite poset with a minimal and a maximal element.  
    The \emph{chain lattice} $\mathcal{C}_P$ is defined as the lattice
    with a top element $\widehat{\mathbf{1}}$
    whose other elements are all chains of~$P$ that contain the minimal and maximal elements of~$P$, ordered by containment.
    The \emph{cyclic chain lattice} of a matroid~$\M$ is
    $\mathcal{C}_{\mathscr{Z}}(\M) := \mathcal{C}_{\mathscr{Z}(\M)}$,
    where $\mathscr{Z}(\M)$ is the lattice of cyclic flats of~$\M$.
\end{definition}

Consider the M\"obius function of $\mathcal{C}_{\mathscr{Z}}(\M)$, as in Stanley \cite[Chapter~3]{stanley-ec1}. To each element $\mathrm{C}\in \mathcal{C}_{\mathscr{Z}}(\M)$ we can associate the number $\lambda_{\mathrm{C}} = -\mu(\mathrm{C},\widehat{\mathbf{1}})$.
Since each of these elements $\mathrm{C}\ne\widehat{\mathbf{1}}$ is a chain of cyclic flats, there is a unique Schubert matroid $\mathsf{S}_{\mathrm{C}}$ whose lattice of cyclic flats coincides with $\mathrm{C}$.

\begin{theorem}\label{thm:indicator-functions-schubert}
    Let $\M$ be a matroid (possibly with loops and coloops). Then,
    \begin{equation}\label{eq:hampe}
    \indicator{\mathscr{P}(\M)} = \sum_{\substack{\mathrm{C}\in\mathcal{C}_{\mathscr{Z}}(\M)\\\mathrm{C} \neq \widehat{\mathbf{1}}}} \lambda_{\mathrm{C}}\, \indicator{\mathscr{P}(\mathsf{S}_{\mathrm{C}})}.
    \end{equation}
\end{theorem}

\begin{proof}
    This statement is proved for loopless and coloopless matroids as \cite[Theorem~5.2]{ferroni}. 
    The general statement follows by Remark~\ref{remark:loops-coloops}.
\end{proof}

\subsection{Isomorphism classes and valuative invariants}
Often one does not wish to distinguish isomorphic matroids.
The story of valuative matroid functions extends nicely to this setting.

\begin{definition}
    Let $X$ be a set and $A$ be an abelian group.
    \begin{itemize}
        \item An \emph{invariant} on matroids on $n$ elements of rank $r$ is a map $f:\mathcal{M}_{r,n}\to X$.
        \item An invariant $f:\mathcal{M}_{r,n}\to A$ is said to be \emph{valuative} if the function $\overline{f}:\overline{\mathcal{M}}_{r,n}\to A$ on labelled matroids defined by $\overline{f}(\M) = f([\M])$ is a valuation.
    \end{itemize}
\end{definition}

\begin{remark}
    Using Theorem~\ref{thm:DF} one can check that if a valuative function $\overline{f}:\overline{\mathcal{M}}_{r,n}\to A$ satisfies the property that $\overline{f}(\mathsf{S}_1) = \overline{f}(\mathsf{S}_2)$ whenever $\mathsf{S}_1$ and $\mathsf{S}_2$ are isomorphic Schubert matroids, then it descends to a valuative invariant $f:\mathcal{M}_{r,n}\to A$.
\end{remark}

The action of the symmetric group $\mathfrak{S}_n$ on~$[n]$
induces an action on~$\overline{\mathcal{M}}_{r,n}$,
of which $\mathcal{M}_{r,n}$ is the set of orbits.
We write this as a left action,
as well as the action of $\mathfrak{S}_n$ on~$\mathbb{R}^n$ by permuting coordinates.
Taking indicator function of the base polytope intertwines these two actions:
\[\indicator{\mathscr{P}(\sigma\M)}(x) = \indicator{\mathscr{P}(\M)}(\sigma x).\]

\begin{definition}\label{def:symindicator}
Let $X$ be a subset of $\mathbb{R}^n$.
Its \emph{symmetrized indicator function} $\symindicator{X}:\mathbb{R}^n\to\mathbb{Q}$ is defined by
\[\symindicator{X}(x) = \frac1{n!}\sum_{\sigma\in \mathfrak{S}_n}\indicator{X}(\sigma x).\]
\end{definition}
Therefore $\symindicator{\mathscr{P}(\sigma\M)} = \symindicator{\mathscr{P}(\M)}$ for any $\M\in\overline{\mathcal{M}}_{r,n}$ and $\sigma\in \mathfrak{S}_n$,
implying that $\symindicator{\mathscr{P}(\M)}$
is well defined given only the isomorphism class $[\M]$.

\begin{corollary}\label{coro:isomorphism-classes-schubert}
    Let $\M$ be a matroid (possibly with loops and coloops). Then,
    \begin{equation}\label{eq:hampe-iso-classes}
    \symindicator{\mathscr{P}(\M)} = \sum_{\substack{\mathrm{C}\in\mathcal{C}_{\mathscr{Z}}(\M)\\\mathrm{C} \neq \widehat{\mathbf{1}}}} \lambda_{\mathrm{C}}\, \symindicator{\mathscr{P}(\mathsf{S}_{\mathrm{C}})}.
    \end{equation}
\end{corollary}

We emphasize that the linear combination appearing in equation \eqref{eq:hampe-iso-classes}, after grouping the terms that correspond to isomorphic Schubert matroids, is still an integer combination.

\section{The polytopes \texorpdfstring{$\overline{\Omega}_{r,n}$}{barOmega(r,n)} and \texorpdfstring{$\Omega_{r,n}$}{Omega(r,n)}}\label{sec:three}

\subsection{The definition of the polytopes} In light of Theorem~\ref{thm:indicator-functions-schubert}, we can associate to each matroid $\M \in \overline{\mathcal{M}}_{r,n}$ a point $\overline{p}_{\M}\in \mathbb{Z}^{\overline{\mathcal{S}}_{r,n}}$, where the coordinate with index $\mathsf{S}_{\mathrm{C}}$ is the integer $\lambda_{\mathrm{C}}$, potentially zero, on the right hand side of equation \eqref{eq:hampe}. 

Similarly, to each isomorphism class $[\M] \in \mathcal{M}_{r,n}$ we can associate a point in $p_{[\M]}\in \mathbb{Z}^{\mathcal{S}_{r,n}}$ using Corollary~\ref{coro:isomorphism-classes-schubert}. 
Notice that now the coordinates are sums of the $\lambda_{\mathrm{C}}$ over isomorphism classes of chains $\mathrm{C}$ under the symmetric group action. 
Two chains are in the same isomorphism class if they consist of sets of the same sizes.

For $\M\neq\N$ the points $\overline{p}_{\M}$ and $\overline{p}_{\N}$ are in fact different, since the sums in \eqref{eq:hampe} give distinct indicator functions $\indicator{\mathscr{P}(\M)}\neq\indicator{\mathscr{P}(\N)}$.
The counterpart for the $p_{[\M]}$ is false: see Remark~\ref{rem:nonisomorphic-same-vertex}.

\begin{definition}
    For each $n$ and $r$, we define \emph{the polytope of all matroids} 
    \[\overline{\Omega}_{r,n} := \operatorname{conv}\left( \overline{p}_{\M} : \M \in \overline{\mathcal{M}}_{r,n}\right) \subseteq \mathbb{R}^{\overline{\mathcal{S}}_{r,n}}, \]
    and \emph{the polytope of all matroids up to isomorphism}
    \[\Omega_{r,n} := \operatorname{conv}\left( p_{[\M]} : [\M] \in \mathcal{M}_{r,n}\right)\subseteq \mathbb{R}^{\mathcal{S}_{r,n}} .\]
\end{definition}

\begin{remark}\label{rem:symmetrization-is-a-projection}
    Notice that the polytope $\Omega_{r,n}$ is a linear projection of the polytope $\overline{\Omega}_{r,n}$ via the map $\pi:\mathbb{R}^{\overline{\mathcal{S}}_{r,n}}\to \mathbb{R}^{\mathcal{S}_{r,n}}$ defined on coordinates by $\overline{p}_{\mathsf{S}}\longmapsto p_{[\mathsf{S}]}$. Clearly, for every matroid $\M\in\overline{\mathcal{M}}_{r,n}$ we have that $\overline{p}_{\M} \stackrel{\pi}{\longmapsto} p_{[\M]}$.
\end{remark}

\begin{example}\label{example:schuberts-4-2}
    Let us describe the polytope $\Omega_{2,4}$. It is known that there are exactly $7$ elements in $\mathcal{M}_{2,4}$ (cf.\ \cite[Table~1]{mayhew-royle}). Among these $7$ isomorphism classes, $\binom{4}{2}=6$ of them comprise Schubert matroids. They are represented in Figure~\ref{fig:schuberts-4-2} as lattice path matroids with a highlighted upper path:
    \begin{center}
    \begin{figure}[ht]
        \begin{tikzpicture}[scale=0.65, line width=.5pt]
        
        \draw[line width=2pt,line cap=round] (0,0) -- (2,0) -- (2,2);
        
        \draw (0,0) grid (2,2);
        \end{tikzpicture}\qquad
        \begin{tikzpicture}[scale=0.65, line width=.5pt]
        
        \draw[line width=2pt,line cap=round] (0,0) -- (1,0) -- (1,1) -- (2,1) -- (2,2);
        
        \draw (0,0) grid (2,2);
        \end{tikzpicture}\qquad
        \begin{tikzpicture}[scale=0.65, line width=.5pt]
        
        \draw[line width=2pt,line cap=round] (0,0) -- (1,0) -- (1,2) -- (2,2);
        
        \draw (0,0) grid (2,2);
        \end{tikzpicture}\qquad
        \begin{tikzpicture}[scale=0.65, line width=.5pt]
        
        \draw[line width=2pt,line cap=round] (0,0) -- (0,1) -- (2,1) -- (2,2);
        
        \draw (0,0) grid (2,2);
        \end{tikzpicture}\qquad
        \begin{tikzpicture}[scale=0.65, line width=.5pt]
        
        \draw[line width=2pt,line cap=round] (0,0) -- (0,1) -- (1,1) -- (1,2) -- (2,2);
        
        \draw (0,0) grid (2,2);
        \end{tikzpicture}\qquad
        \begin{tikzpicture}[scale=0.65, line width=.5pt]
        
        \draw[line width=2pt,line cap=round] (0,0) -- (0,2) -- (2,2);
        
        \draw (0,0) grid (2,2);
        \end{tikzpicture}
    \caption{The six (isomorphism classes of) Schubert matroids in $\mathcal{S}_{2,4}$.}\label{fig:schuberts-4-2}
    \end{figure}
    \end{center}
    The only non-Schubert isomorphism class corresponds to $\U_{1,2}\oplus \U_{1,2}$, the direct sum of two uniform matroids of rank $1$ and size $2$. The computation in Corollary~\ref{coro:isomorphism-classes-schubert} yields the following equality:
    \[ \symindicator{\U_{1,2}\oplus \U_{1,2}} = 
        2 \cdot \symindicator{
        \begin{tikzpicture}[baseline=1.5mm,scale=0.30, line width=.5pt]
        \vspace*{10pt}
        \draw[line width=2pt,line cap=round] (0,0) -- (0,1) -- (1,1) -- (1,2) -- (2,2);
        
        \draw (0,0) grid (2,2);
        \end{tikzpicture}}
        - 1\cdot\symindicator{
        \begin{tikzpicture}[baseline=1.5mm,scale=0.30, line width=.5pt]
        
        \draw[line width=2pt,line cap=round] (0,0) -- (0,2) -- (2,2);
        
        \draw (0,0) grid (2,2);
        \end{tikzpicture}}
    \]
    If we label the six Schubert matroids in Figure~\ref{fig:schuberts-4-2} from left to right as coordinates $1$ through~$6$, the polytope $\Omega_{2,4}$ is the convex hull of the seven columns of the following matrix:
    \[\begin{bmatrix}
        1 & 0 & 0 & 0 & 0 & 0 & 0\\
        0 & 1 & 0 & 0 & 0 & 0 & 0\\
        0 & 0 & 1 & 0 & 0 & 0 & 0\\
        0 & 0 & 0 & 1 & 0 & 0 & 0\\
        0 & 0 & 0 & 0 & 1 & 0 & 2\\
        0 & 0 & 0 & 0 & 0 & 1 & -1\\
    \end{bmatrix}.\]
    So the polytope $\Omega_{2,4}$ is a simplex. 
    Only the fifth column is not a vertex, 
    and it is the midpoint of the edge between the sixth and seventh columns.
\end{example}

The next technical lemma says that the point $p_{[\M]}$ and the symmetrized indicator function $\symindicator{\mathscr{P}(\M)}$ encode the same information, and that the ambient space $\mathbb{R}^{\mathcal{S}_{r,n}}$ of the polytope $\Omega_{r,n}$ is naturally isomorphic to the linear span of the symmetrized indicator functions of all matroids in $\mathcal{M}_{r,n}$.  

\begin{lemma}\label{lemma:symindicator}
The real vector space $V_{r,n}$ spanned by the functions $\symindicator{\mathscr{P}(\M)}$ for $\M\in\mathcal{M}_{r,n}$
is isomorphic to $\mathbb{R}^{\overline{\mathcal{S}}_{r,n}}$,
by an isomorphism which sends $\symindicator{\mathscr{P}(\M)}$ to $p_{[\M]}$ for all $\M\in\overline{\mathcal{M}}_{r,n}$.
\end{lemma}

\begin{proof}
Define a linear function 
$\phi:\mathbb{R}^{\overline{\mathcal{S}}_{r,n}}\to V_{r,n}$
by prescribing its values on the basis $\{p_{[\mathsf{S}]}:[\mathsf{S}]\in\mathcal{S}_{r,n}\}$ of the domain:
$\phi(p_{[\mathsf{S}]})=\symindicator{\mathscr{P}(\mathsf{S})}$.
By definition of $p_{[\M]}$, in fact 
$\phi(p_{[\M]})=\symindicator{\mathscr{P}(\M)}$ for all $\M$.

The functions called $s^{\mathrm{sym}}_{\underline X,\underline r}$ in \cite[Corollary 6.4]{derksen-fink}
are sums of $\mathfrak{S}_n$-orbits of valuations.
Expanding this sum and rewriting to move the $\sigma\in \mathfrak{S}_n$ from acting on the valuation
to acting on the indicator function,
we get an expression for $s^{\mathrm{sym}}_{\underline X,\underline r}(\M)$ as a linear function on $V_{r,n}$ applied to $\symindicator{\mathscr{P}(\M)}$.
Thus $s^{\mathrm{sym}}_{\underline X,\underline r}$ is a linear function on $V_{r,n}$.
Now the basis $\mathcal{S}_{r,n}$ of matroid isomorphism classes mod valuative relations
has a dual basis $\{a_{[\mathsf{S}]}:[\mathsf{S}]\in\mathcal S_{r,n}\}$ of valuative invariants;
these are linear combinations of the $s^{\mathrm{sym}}_{\underline X,\underline r}$,
and therefore they are also linear functions on $V_{r,n}$.
An inverse to $\phi$ is therefore given as
$\sum_{\mathsf S} a_{[\mathsf{S}]}p_{[\mathsf{S}]}$.
\end{proof}

\begin{remark}\label{rem:nonisomorphic-same-vertex}
    Two distinct isomorphism classes $[\M]$ and $[\N]$ in $\mathcal{M}_{r,n}$ may have the same symmetrized indicator function, and thus yield the same point $p_{[\M]}=p_{[\N]}\in \mathbb{R}^{\mathcal{S}_{r,n}}$. For example, consider the matroids with ground set $\{1,\ldots,6\}$ of rank $3$ and having set of bases:
    \begin{align*}
        \mathscr{B}(\M) &= \binom{[6]}{3} \smallsetminus \{\{1,2,3\}, \{4,5,6\}\},\\
        \mathscr{B}(\N) &= \binom{[6]}{3} \smallsetminus \{\{1,2,3\}, \{3,4,5\}\}.
    \end{align*}
    They are not isomorphic, yet Corollary~\ref{coro:isomorphism-classes-schubert} gives:
    \[p_{[\M]} = p_{[\N]} = 2\cdot 
        \begin{tikzpicture}[baseline=2.6mm,scale=0.25, line width=.5pt]
        \vspace*{10pt}
        \draw[line width=2pt,line cap=round] (0,0) -- (0,2) -- (1,2) -- (1,2) -- (1,3) -- (3,3);
        
        \draw (0,0) grid (3,3);
        \end{tikzpicture}
        - 1\cdot
        \begin{tikzpicture}[baseline=2.6mm,scale=0.25, line width=.5pt]
        
        \draw[line width=2pt,line cap=round] (0,0) -- (0,3) -- (3,3);
        
        \draw (0,0) grid (3,3);
        \end{tikzpicture},
    \]
    where on the right hand side the diagram for a Schubert matroid $\mathsf{S}$ means $p_{[\mathsf{S}]}\in\mathbb{R}^{\mathcal{S}_{r,n}}$.
\end{remark}

\subsection{The dimension of the polytopes} The first question one might ask of the two polytopes $\Omega_{r,n}$ and $\overline{\Omega}_{r,n}$ is their dimension.

\begin{theorem}\label{thm:dimensions-of-polytopes}
    The polytopes $\Omega_{r,n}$ and $\overline{\Omega}_{r,n}$ have codimension $1$. In other words,
        \begin{align*}
            \dim \Omega_{r,n} &= |\mathcal{S}_{r,n}| - 1,\\
            \dim \overline{\Omega}_{r,n} &= |\overline{\mathcal{S}}_{r,n}| - 1.
        \end{align*}
\end{theorem}

\begin{proof}
    Let us see this for the polytope $\overline{\Omega}_{r,n}$, as the proof for the other is almost identical. Every Schubert matroid in $\overline{\mathcal{S}}_{r,n}$ corresponds to a standard basis vector in the ambient space $\mathbb{R}^{\overline{\mathcal{S}}_{r,n}}$.
    This tells us that the dimension of $\overline{\Omega}_{r,n}$ is at least $|\overline{\mathcal{S}}_{r,n}|-1$. On the other hand, every point in $\overline{\Omega}_{r,n}$ has sum of coordinates equal to $1$, so the codimension is at least $1$. To prove the last assertion, it suffices to show that for every $\M\in \overline{\mathcal{M}}_{r,n}$ one has
    \[\sum_{\substack{\mathrm{C}\in\mathcal{C}_{\mathscr{Z}}(\M)\\\mathrm{C} \neq \widehat{\mathbf{1}}}} \lambda_{\mathrm{C}} = 1.\]
    This is immediately true either by the defining recursion of the M\"obius function, or by using Example~\ref{ex:constants-are-valuations} together with equation~\eqref{eq:hampe}.
\end{proof}

As mentioned earlier, Schubert matroids $\mathsf{S}\in \mathcal{S}_{r,n}$ can be seen up to isomorphism as the lattice path matroids whose lower path is the lower-right border of a grid. Since there are no restrictions on the upper path, 
we have
\[|\mathcal{S}_{r,n}| = \binom{n}{r}.\]
On the other hand, the number of (labelled) Schubert matroids on $n$ elements of rank $r$ is a 
\emph{binomial Eulerian number}, by \cite[Theorem~1.1]{stellahedral} plus, e.g., \cite[Section~10.4]{postnikov-reiner-williams}; 
see also the generating function in \cite[Theorem 1.5c]{derksen-fink}. 
We have
\[
        |\overline{\mathcal{S}}_{r,n}| = \sum_{j=0}^{n-r} \binom{n}{r+j} A_{r+j,j},
\]
where the $A_{m,i}$ are the usual \emph{Eulerian numbers}, i.e., $A_{m,i}$ is the number of permutations on $m$ elements that have exactly $i$ descents (see \cite{petersen}). This makes clear that the polytopes $\Omega_{r,n}$ and $\overline{\Omega}_{r,n}$ live in large-dimensional Euclidean spaces, and therefore their direct computation becomes infeasible very quickly, even for small values of $n$ and $r$. As an example, $\dim \Omega_{3,6} = 19$ and $\dim \overline{\Omega}_{3,6} = 882$.

\subsection{Integral points of the polytopes}

Our polytopes encode matroids through their lattice points, so it is natural to ask whether all the lattice points of our polytopes necessarily come from matroids. 

\begin{theorem}
    The only lattice points in $\overline{\Omega}_{r,n}$ are its vertices.
\end{theorem}

\begin{proof}
    Let us denote by $Z$ any fixed chain of sets of the form $\varnothing\subseteq Z_0\subsetneq\cdots\cdot\subsetneq Z_m \subseteq [n]$. Denote by $\mathbf{r}$ any fixed sequence of integers $r_0,\ldots,r_m$. Consider the map $\Phi_{Z,\mathbf{r}}:\overline{\mathcal{M}}_{r,n}\to \mathbb{R}$ taking values $0$ or $1$ according to whether each $Z_i$ has rank $r_i$ in $\M$. By \cite[Proposition~5.3]{derksen-fink}, this map is a valuation. Consider the set $U$ of all possible pairs $(Z,\mathbf{r})$ as above, and the valuative function $\Phi:\mathcal{M}_{r,n} \to \mathbb{R}^{U}$ whose coordinates are given by each $\Phi_{Z,\mathbf{r}}$ as before. By \cite[Corollary~5.6]{derksen-fink}, the valuations of the form $\Phi_{Z,\mathbf{r}}$ span the space of valuative functions on $\mathcal{M}_{r,n}$. Therefore, the function $\Phi$ can be viewed as an injective map $\Phi:\mathbb{R}^{\overline{\mathcal{S}}_{r,n}} \to \mathbb{R}^{U}$ that sends every lattice point of $\mathbb{R}^{\overline{\mathcal{S}}_{r,n}}$ to some lattice point in $\mathbb{R}^U$. However, since every vertex of $\overline{\Omega}_{r,n}$ is sent to a $\{0,1\}$-vector in $\mathbb{R}^{U}$, the image of $\overline{\Omega}_{r,n}$ under this map is a subpolytope of the unit cube in $\mathbb{R}^U$, and hence its only lattice points are its vertices. This immediately implies that $\overline{\Omega}_{r,n}$ does not contain other lattice points than its vertices.
\end{proof}

A slight modification of the strategy employed in the last proof was suggested to us by Benjamin Schr\"oter in private communication. The above proof in fact also yields that every point $\overline{p}_{\M}$ for $\M\in \mathcal{M}_{r,n}$ is a vertex of $\overline{\Omega}_{r,n}$, a result that we will prove below, in Theorem~\ref{thm:vertices-omegabar}, using a much simpler approach that does not rely on the valuativity of the map $\Phi$.

Notice that the same argument used in the last proof will not apply for the polytope $\Omega_{r,n}$, because we know that it contains lattice points other than its vertices (see Example~\ref{example:schuberts-4-2}, and Example~\ref{ex:T24-non-extremal} in the next section). However, a natural question to ask is whether all the lattice points in $\Omega_{r,n}$ do come from matroids. The answer to that question is negative, as can be verified with the use of a computer.

\begin{theorem}\label{thm:lattice-points-not-matroids}
    The polytope $\Omega_{3,6}$ contains $5$ lattice points which do not represent matroids.
\end{theorem}

For the interested reader, we indicate below what the $5$ lattice points are in the Schubert basis.

\begin{itemize}[itemsep=10pt]
\item[] $p_1 = 2\cdot 
    \begin{tikzpicture}[baseline=2.4mm,scale=0.30, line width=.5pt]
        \vspace*{10pt}
        \draw[line width=2pt,line cap=round] (0,0) -- (0,1) -- (1,1) -- (1,2) -- (2,2) -- (2,3) -- (3,3);
        
        \draw (0,0) grid (3,3);
        \end{tikzpicture} - \begin{tikzpicture}[baseline=2.4mm,scale=0.30, line width=.5pt]
        \vspace*{10pt}
        \draw[line width=2pt,line cap=round] (0,0) -- (0,2) -- (2,2) -- (2,3) -- (3,3);
        
        \draw (0,0) grid (3,3);
        \end{tikzpicture} - \begin{tikzpicture}[baseline=2.4mm,scale=0.30, line width=.5pt]
        \vspace*{10pt}
        \draw[line width=2pt,line cap=round] (0,0) -- (0,1) -- (1,1) -- (1,3) -- (3,3);
        
        \draw (0,0) grid (3,3);
        \end{tikzpicture} + \begin{tikzpicture}[baseline=2.4mm,scale=0.30, line width=.5pt]
        \vspace*{10pt}
        \draw[line width=2pt,line cap=round] (0,0) -- (0,3) -- (3,3);
        
        \draw (0,0) grid (3,3);
        \end{tikzpicture} $,
    \item[] $p_2 = 2\cdot \begin{tikzpicture}[baseline=2.4mm,scale=0.30, line width=.5pt]
        \vspace*{10pt}
        \draw[line width=2pt,line cap=round] (0,0) -- (0,1) -- (1,1) -- (1,2) -- (2,2) -- (2,3) -- (3,3);
        
        \draw (0,0) grid (3,3);
        \end{tikzpicture} - \begin{tikzpicture}[baseline=2.4mm,scale=0.30, line width=.5pt]
        \vspace*{10pt}
        \draw[line width=2pt,line cap=round] (0,0) -- (0,2) -- (2,2) -- (2,3) -- (3,3);
        
        \draw (0,0) grid (3,3);
        \end{tikzpicture} - \begin{tikzpicture}[baseline=2.4mm,scale=0.30, line width=.5pt]
        \vspace*{10pt}
        \draw[line width=2pt,line cap=round] (0,0) -- (0,1) -- (1,1) -- (1,3) -- (3,3);
        
        \draw (0,0) grid (3,3);
        \end{tikzpicture} + \begin{tikzpicture}[baseline=2.4mm,scale=0.30, line width=.5pt]
        \vspace*{10pt}
        \draw[line width=2pt,line cap=round] (0,0) -- (0,2) -- (1,2) -- (1,2) -- (1,3) -- (3,3);
        
        \draw (0,0) grid (3,3);
        \end{tikzpicture}$,
    \item[] $p_3 = 2\cdot \begin{tikzpicture}[baseline=2.4mm,scale=0.30, line width=.5pt]
        \vspace*{10pt}
        \draw[line width=2pt,line cap=round] (0,0) -- (0,1) -- (1,1) -- (1,2) -- (2,2) -- (2,3) -- (3,3);
        
        \draw (0,0) grid (3,3);
        \end{tikzpicture} - \begin{tikzpicture}[baseline=2.4mm,scale=0.30, line width=.5pt]
        \vspace*{10pt}
        \draw[line width=2pt,line cap=round] (0,0) -- (0,2) -- (2,2) -- (2,3) -- (3,3);
        
        \draw (0,0) grid (3,3);
        \end{tikzpicture} - \begin{tikzpicture}[baseline=2.4mm,scale=0.30, line width=.5pt]
        \vspace*{10pt}
        \draw[line width=2pt,line cap=round] (0,0) -- (0,1) -- (1,1) -- (1,3) -- (3,3);
        
        \draw (0,0) grid (3,3);
        \end{tikzpicture} + 2\cdot \begin{tikzpicture}[baseline=2.4mm,scale=0.30, line width=.5pt]
        \vspace*{10pt}
        \draw[line width=2pt,line cap=round] (0,0) -- (0,2) -- (1,2) -- (1,2) -- (1,3) -- (3,3);
        
        \draw (0,0) grid (3,3);
        \end{tikzpicture} -  \begin{tikzpicture}[baseline=2.4mm,scale=0.30, line width=.5pt]
        \vspace*{10pt}
        \draw[line width=2pt,line cap=round] (0,0) -- (0,3) -- (3,3);
        
        \draw (0,0) grid (3,3);
        \end{tikzpicture}$,
    \item[]  $p_4 = 4\cdot 
    \begin{tikzpicture}[baseline=2.4mm,scale=0.30, line width=.5pt]
        \vspace*{10pt}
        \draw[line width=2pt,line cap=round] (0,0) -- (0,1) -- (1,1) -- (1,2) -- (2,2) -- (2,3) -- (3,3);
        
        \draw (0,0) grid (3,3);
        \end{tikzpicture} -2 \cdot \begin{tikzpicture}[baseline=2.4mm,scale=0.30, line width=.5pt]
        \vspace*{10pt}
        \draw[line width=2pt,line cap=round] (0,0) -- (0,2) -- (2,2) -- (2,3) -- (3,3);
        
        \draw (0,0) grid (3,3);
        \end{tikzpicture} - 2\cdot \begin{tikzpicture}[baseline=2.4mm,scale=0.30, line width=.5pt]
        \vspace*{10pt}
        \draw[line width=2pt,line cap=round] (0,0) -- (0,1) -- (1,1) -- (1,3) -- (3,3);
        
        \draw (0,0) grid (3,3);
        \end{tikzpicture} + \begin{tikzpicture}[baseline=2.4mm,scale=0.30, line width=.5pt]
        \vspace*{10pt}
        \draw[line width=2pt,line cap=round] (0,0) -- (0,3) -- (3,3);
        
        \draw (0,0) grid (3,3);
        \end{tikzpicture}$,
    \item[] $p_5 = 4\cdot \begin{tikzpicture}[baseline=2.4mm,scale=0.30, line width=.5pt]
        \vspace*{10pt}
        \draw[line width=2pt,line cap=round] (0,0) -- (0,1) -- (1,1) -- (1,2) -- (2,2) -- (2,3) -- (3,3);
        
        \draw (0,0) grid (3,3);
        \end{tikzpicture} - 2\cdot \begin{tikzpicture}[baseline=2.4mm,scale=0.30, line width=.5pt]
        \vspace*{10pt}
        \draw[line width=2pt,line cap=round] (0,0) -- (0,2) -- (2,2) -- (2,3) -- (3,3);
        
        \draw (0,0) grid (3,3);
        \end{tikzpicture} - 2\cdot \begin{tikzpicture}[baseline=2.4mm,scale=0.30, line width=.5pt]
        \vspace*{10pt}
        \draw[line width=2pt,line cap=round] (0,0) -- (0,1) -- (1,1) -- (1,3) -- (3,3);
        
        \draw (0,0) grid (3,3);
        \end{tikzpicture} + \begin{tikzpicture}[baseline=2.4mm,scale=0.30, line width=.5pt]
        \vspace*{10pt}
        \draw[line width=2pt,line cap=round] (0,0) -- (0,2) -- (1,2) -- (1,2) -- (1,3) -- (3,3);
        
        \draw (0,0) grid (3,3);
        \end{tikzpicture}$.
    \end{itemize}

\section{Vertices of the polytopes: Extremal matroids}\label{sec:four}

Now that we have defined our polytopes $\overline{\Omega}_{r,n}$ and $\Omega_{r,n}$, it is natural to ask about their facial structures. As we will see in this section, their sets of vertices behave differently. In the case of $\overline{\Omega}_{r,n}$ every labelled matroid yields a distinct vertex. For the polytope $\Omega_{r,n}$ the situation is considerably subtler: some isomorphism classes of matroids fall together in the polytope, and not all isomorphism classes correspond to vertices. The following is a direct consequence of our definition of the polytopes $\overline{\Omega}_{r,n}$ and $\Omega_{r,n}$.

\begin{proposition}
    The linear functions on the ambient space of $\overline{\Omega}_{r,n}$ are the valuative functions on $\overline{\mathcal{M}}_{r,n}$. The linear functions on the ambient space of $\Omega_{r,n}$ are the valuative invariants on $\mathcal{M}_{r,n}$. 
\end{proposition}

Therefore, characterizing vertices and, more generally, faces of these polytopes
is a matter of finding valuative functions or invariants
that attain their maximum exactly at prescribed points 
$\overline{p}_{\M}$, respectively $p_{[\M]}$.

\subsection{The vertices of \texorpdfstring{$\overline{\Omega}_{r,n}$}{barOmega(r,n)}}

\begin{theorem}\label{thm:vertices-omegabar}
    For every matroid $\M\in\overline{\mathcal{M}}_{r,n}$, 
    the point $\overline{p}_{\M}$ is a vertex of $\overline{\Omega}_{r,n}$. 
\end{theorem}

\begin{proof}
    Let us fix $\M$ and consider its sets of bases $\mathscr{B}(\M)$ and of non-bases $\mathscr{B}(\M)^{\mathrm{c}} = \binom{[n]}{r} \smallsetminus \mathscr{B}(\M)$. Define the function $g:\overline{\mathcal{M}}_{r,n}\to \mathbb{Z}$ by 
    \[ g(\N) = \sum_{B\in\mathscr{B}(\M)} f_B(\N) + \sum_{B\in \mathscr{B}(\M)^{\mathrm{c}}} (1-f_B(\N)).\]
    where $f_B$ is the basis indicator function defined in Example~\ref{ex:compendium}. By Example~\ref{ex:constants-are-valuations}, $g$ is a sum of valuations and therefore a valuation itself. Notice that $g(\M) = \binom{n}{r}$ while $g(\N) < \binom{n}{r}$ for every matroid $\N\in \overline{\mathcal{M}}_{r,n}$ such that $\N\neq \M$. This shows that $\M$ corresponds to a vertex of $\overline{\Omega}_{r,n}$.
\end{proof}

\subsection{The vertices of \texorpdfstring{$\Omega_{r,n}$}{Omega(r,n)} and first examples of extremal matroids}

We now turn our attention to the vertices of the polytopes $\Omega_{r,n}$. 
Their subtler behaviour motivates us to name the matroids lying at these vertices.

\begin{definition}
    A matroid $\M$ of rank $r$ on $n$ elements is said to be \emph{extremal} if $p_{[\M]}$ is a vertex of the polytope $\Omega_{r,n}$.
\end{definition}

We once again put emphasis on the fact that vertices of $\Omega_{r,n}$ correspond a~priori to \emph{sets} of isomorphism classes of matroids: 
Remark~\ref{rem:nonisomorphic-same-vertex} shows that distinct elements of $\mathcal{M}_{r,n}$ can produce the same point in $\mathbb{R}^{\mathcal{S}_{r,n}}$. In particular, if $p_{[\M]} = p_{[\N]}$ for two different isomorphism classes $\M$ and $\N$, we have that $\M$ is extremal if and only if $\N$ is extremal. See Table~\ref{table:number-matroids} where we list on the left the number of vertices of $\Omega_{r,n}$ and on the right the number of classes of isomorphism of extremal matroids on $[n]$ having rank $r$, for $0\leq r\leq n\leq 8$. Observe that for $n=8$ and $r=3,5$ the numbers in the two tables differ, which means that in the polytopes $\Omega_{3,8}$ and $\Omega_{5,8}$ there are vertices that correspond to more than one isomorphism class of matroids. We have been unable to compute the numbers for $(r,n)=(4,8)$.

{\footnotesize
\begin{table}\label{table:}
    \begin{subtable}[ht]{.5\textwidth}
        \begin{tabular}{l r r r r r r r r}\hline
        $r\backslash n$   & 1   & 2 & 3 & 4 & 5 & 6 & 7& 8   \\ \hline
        0 &   1 & 1 & 1 & 1 & 1 & 1 & 1 & 1\\
        1 &   1 & 2 & 3 & 4 & 5 & 6 & 7 & 8\\
        2 &    & 1 & 3 & 6 & 11 & 17 & 27 & 38\\
        3 &    &  & 1 & 4 & 11 & 28 & 64  & 145\\
        4 &    &  &  & 1 & 5 & 17 & 64 & ? \\
        5 &    &  &  &  & 1 & 6 & 27 & 145\\
        6 &    &  &  &  &  & 1 & 7 & 38\\
        7 &    &  &  &  &  &  & 1  & 8 \\
        8 &    &  &  &  &  &  &    & 1
        \end{tabular}
    \end{subtable}
    \begin{subtable}[t]{.5\textwidth}
        
        \begin{tabular}{l r r r r r r r r}\hline
        $r\backslash n$   & 1   & 2 & 3 & 4 & 5 & 6 & 7& 8   \\ \hline
        0 &   1 & 1 & 1 & 1 & 1 & 1 & 1 & 1\\
        1 &   1 & 2 & 3 & 4 & 5 & 6 & 7 & 8\\
        2 &    & 1 & 3 & 6 & 11 & 17 & 27 & 38\\
        3 &    &  & 1 & 4 & 11 & 28 & 64  & 148\\
        4 &    &  &  & 1 & 5 & 17 & 64 & ? \\
        5 &    &  &  &  & 1 & 6 & 27 & 148\\
        6 &    &  &  &  &  & 1 & 7 & 38\\
        7 &    &  &  &  &  &  & 1  & 8 \\
        8 &    &  &  &  &  &  &    & 1
        \end{tabular}
    \end{subtable}
    \caption{On the left the number of vertices of $\Omega_{r,n}$. On the right the number of isomorphism classes of extremal matroids.}
    \label{table:number-matroids}
\end{table}
}

We will be focusing on isomorphism classes of matroids throughout this section, but we will often write $\M$ instead of $[\M]$ to avoid overloading the notation. The following result is our workhorse tool for showing that given isomorphism classes of matroids are indeed extremal.

\begin{lemma}\label{lemma:extremal-sequence-invariants}
    Let $f_1,\ldots,f_m$ be a sequence of valuative invariants on $\mathcal{M}_{r,n}$. Consider the following sets of isomorphism classes of matroids:
        \begin{align*}
        S_0 &= \mathcal{M}_{r,n},\\
        S_1 &= \left\{\M\in S_0 \;: f_1(\M) = \max_{\N\in S_0} f_1(\N)\right\},\\
        S_2 &= \left\{\M\in S_1 \;: f_2(\M) = \max_{\N\in S_1} f_2(\N)\right\},\\
        &\enspace \vdots\\
        S_m &= \left\{\M\in S_{m-1} : f_m(\M) = \max_{\N\in S_{m-1}} f_m(\N)\right\}.
        \end{align*}
    If all the elements in $S_m$ yield the same point $p_{[\M]}\in \mathbb{R}^{\mathcal{S}_{r,n}}$, then $p_{[\M]}$ is a vertex and all the elements in $S_m$ correspond to isomorphism classes of extremal matroids.
\end{lemma}

\begin{proof}
    The proof is immediate from the fact that the vertices of any polytope can be characterized as the points that maximize a sequence of linear functionals. In the case of $\Omega_{r,n}$, linear functionals on $\mathbb{R}^{\mathcal{S}_{r,n}}$ correspond to valuative invariants on $\mathcal{M}_{r,n}$.
\end{proof}

Of course, minimizing an invariant $f_i$ is the same as maximizing $-f_i$, and the sign change preserves valuativity.
So when we apply this lemma we can freely choose to minimize valuations as well as to maximize them.

\begin{example}\label{example:uniforms-are-extremal}
    Uniform matroids are extremal. We have a valuative invariant $f:\mathcal{M}_{r,n}\to\mathbb{R}$ given by $\M\mapsto |\mathscr{B}(\M)|$, i.e., the number of bases of the matroid or, equivalently, the number of integer points in the base polytope $\mathscr{P}(\M)$ (see Example~\ref{ex:compendium}). 
    The unique class $\M\in \mathcal{M}_{r,n}$ for which this valuation attains its maximum value is the uniform matroid $\U_{r,n}$, which is therefore a vertex of $\Omega_{r,n}$.
\end{example}

\begin{example}\label{ex:T24-non-extremal}
    Example~\ref{example:schuberts-4-2} shows that the minimal matroid $\mathsf{T}_{2,4}$, which is the Schubert matroid with lattice path drawing
    $
        \begin{tikzpicture}[scale=0.30, line width=.5pt, baseline=2.0mm]
        
        \draw[line width=2pt,line cap=round] (0,0) -- (0,1) -- (1,1) -- (1,2) -- (2,2);
        
        \draw (0,0) grid (2,2);
        \end{tikzpicture},    
    $
    is not extremal\footnote{Recall that minimal matroids minimize the number of bases (which is a valuative invariant) but only among \emph{connected} matroids.}. In fact it is the smallest non-extremal matroid. 
%
%
\end{example}

\begin{example}
    The matroid $\U_{r,r} \oplus \U_{0,n-r} \in \mathcal{M}_{r,n}$ is extremal. Notice that this matroid has only one basis, and that any other isomorphism class of matroids has at least $2$ bases. As was said in Example~\ref{example:uniforms-are-extremal}, the number of bases is a valuative invariant, and since $\U_{r,r}\oplus \U_{0,n-r}$ is the only minimizer up to isomorphism, it is extremal.
\end{example}

\begin{example}\label{example:thickening-uniform}
    To see a less straightforward application of Lemma~\ref{lemma:extremal-sequence-invariants}, 
    consider the matroid $\U_{r,n}^{(2)}\in \mathcal{M}_{r,2n}$ obtained from $\U_{r,n}$ by adding a parallel copy to each element of the ground set. Let us show that when $r > 1$ this matroid is extremal. Consider the valuative invariant
    \begin{align*}
        f_1(\M) = \sum_{\substack{1\le i\le r\\i\neq 2}} \#\{C \text{ circuit of $\M$} : |C| = i\}.
    \end{align*}
    (The valuativity of this invariant follows from Example~\ref{ex:compendium}.)
    The minimum value that $f_1(\M)$ could possibly attain is $0$, and this happens precisely when $\M$ has no loops nor circuits of sizes $3,\ldots,r$. This is equivalent to the simplified matroid $\operatorname{si}(\M)$ not having circuits of sizes $1,\ldots,r$, i.e., to $\operatorname{si}(\M)$ being uniform. In other words, the set $S_1$ defined as in Lemma~\ref{lemma:extremal-sequence-invariants} (with a $\min$ instead of a $\max$) is precisely the set of isomorphism classes of matroids $\M$ whose simplification is uniform. Now, by Example~\ref{ex:compendium} we may also consider the valuative invariant
    \[f_2(\M) := \sum_{\substack{1\le i\le r\\i\neq 2}} \#\{F\in \mathcal{L}(\M): \rk(F) = 1,\, |F| = i\}. \]
    The minimum value that this invariant attains is $0$, and this happens if and only if all the flats of rank $1$ have size $2$. It is immediate to conclude that the set $S_2$ that would result from Lemma~\ref{lemma:extremal-sequence-invariants} has actually only one isomorphism class, and it is precisely $\U_{r,n}^{(2)}$. In particular, this is an extremal matroid.
\end{example}

\subsection{Extremality and direct sums} The class of extremal matroids is not closed under minors: 
for example, the non-extremal $\mathsf{T}_{2,4}$ is a restriction of the extremal $\mathsf{U}^{(2)}_{2,3}$. 
On the other hand, matroid duality does preserve extremality because in $p_{[\M]}$ the coordinate indexed by some $[\mathsf{S}]\in \mathcal{S}_{r,n}$ coincides with the coordinate indexed by $[\mathsf{S}^*]\in \mathcal{S}_{n-r,n}$ in $p_{[\M^*]}$. It is natural to inquire about the preservation of extremality under other basic operations such as direct sums.

\begin{theorem}
    If a direct sum $\M = \M_1\oplus \cdots \oplus \M_s$ is extremal, then each direct summand $\M_1,\ldots,\M_s$ is extremal. 
\end{theorem}

\begin{proof}
By induction it is enough to assume $s=2$. Let $n_i:=|E(\M_i)|$ and let $r_i$ be the rank of~$\M_i$ for $i\in \{1,2\}$. Let us assume, without loss of generality, that $\M_1$ is not extremal. This implies that $p_{[\M_1]}$ is a nontrivial convex combination
\[p_{[\M_1]}=\sum_{j=1}^k b_j\,p_{[\mathsf N_j]}\]
of other distinct points $p_{[\mathsf N_j]}$ in $\Omega_{r,n}$.
By Lemma~\ref{lemma:symindicator}, we have
\[\symindicator{\mathscr{P}(\M_1)}=\sum_{j=1}^k b_j\symindicator{\mathscr{P}(\mathsf{N}_j)}.\]
Taking Cartesian products with $\mathscr{P}(\M_2)$, we have
\[\frac1{n_1!}\sum_{\sigma\in S_{n_1}}\indicator{\mathscr{P}(\sigma(\M_1)\oplus\M_2)}
=\sum_{j=1}^kb_j\cdot\frac1{n_1!}\sum_{\sigma\in S_{n_1}}\indicator{\mathscr{P}(\sigma(\mathsf N_j)\oplus \M_2)}.\]
Now, averaging over a set of coset representatives for $\mathfrak{S}_{n_1}$ in $\mathfrak{S}_{n_1+n_2}$, we see that it is possible
to express $\symindicator{\mathscr{P}(\M_1\oplus\M_2)}$ as a convex combination of other symmetrized indicator functions.
This shows $\M_1\oplus\M_2$ is not extremal.
\end{proof}

We conjecture the converse, which appears subtle:

\begin{conjecture}\label{conj:direct-sums}
    If $\M_1,\ldots,\M_s$ are extremal matroids, then $\M_1\oplus \cdots \oplus \M_s$ is extremal.
\end{conjecture}

We have verified with the help of a computer that a disconnected matroid on at most $7$ elements is extremal if and only if all of its direct summands are extremal.

\section{Special faces}\label{sec:faces}

This section's aim is to show that many famous classes of matroids correspond to actual faces of the polytopes $\Omega_{r,n}$ and $\overline{\Omega}_{r,n}$. 

\begin{definition}
A family of matroids $\overline{\mathcal{F}}\subseteq\overline{\mathcal{M}}_{r,n}$ is \emph{facial in~$\overline{\Omega}_{r,n}$}
if the elements in the class are \emph{exactly the vertices} of some face of $\overline{\Omega}_{r,n}$. 

Similarly, a family of isomorphism classes $\mathcal{F}$ is \emph{facial in~$\Omega_{r,n}$} if the elements in $\overline{\mathcal{F}}$ correspond \emph{exactly to the set of $[\M]$ such that $p_{[\M]}$ is a point} of that face of $\Omega_{r,n}$.
\end{definition}

Note that a~priori we do not rule out the presence of further lattice points that do not come from isomorphism classes of matroids (see Theorem~\ref{thm:lattice-points-not-matroids} below).

\begin{lemma}
    Let $\mathcal{F}\subseteq\mathcal{M}_{r,n}$ a set of isomorphism classes, and let $\overline{\mathcal{F}}\subseteq \overline{\mathcal{M}}_{r,n}$ the set of all matroids whose isomorphism class lies in $\mathcal{F}$. If $\mathcal{F}$ is facial in $\Omega_{r,n}$, then $\overline{\mathcal{F}}$ is facial in $\overline{\Omega}_{r,n}$.
\end{lemma}

\begin{proof}
    This is immediate from the fact that $\Omega_{r,n}$ is a projection of $\overline{\Omega}_{r,n}$. Recall that the preimage of a face under a projection is always a face.
\end{proof}

We will also use the following handy lemma, which is of independent interest for matroid theorists. Recall that the \emph{girth} of a matroid is the size of its smallest circuit.

\begin{lemma}\label{lemma:girth}
    Let $\M\in \overline{\mathcal{M}}_{r,n}$ be a matroid. Denote by $\overline{\mathcal{S}}(\M)$ the set of all the Schubert matroids $\mathsf{S}\in \overline{\mathcal{S}}_{r,n}$ for which the coefficients $a_{\mathsf{S}}$ in equation~\eqref{eq:integer-combination-of-Schuberts} are non-zero. Then,
        \begin{equation}\label{eq:girths}
        \girth(\M) = \min\left\{ \girth(\mathsf{S}) : \mathsf{S} \in \overline{\mathcal{S}}(\M)\right\}.
        \end{equation}
\end{lemma}

\begin{proof}
    Let us name $\alpha := \girth(\M)$, and give the right~hand side of equation~\eqref{eq:girths} the name $\beta$. If we truncate $\M$ exactly $r-\alpha+1$ times, we obtain the uniform matroid $\operatorname{tr}^{(r-\alpha+1)}(\M) = \U_{\alpha,n}$. By relying on Lemma~\ref{lemma:truncation}, the only possibility is that the truncations of all the Schubert matroids with non-zero coefficient are also uniform. This means that all of these Schubert matroids have girth at least $\alpha$. In other words, $\girth(\M) = \alpha \leq \beta$.

    On the other hand, consider the valuative function $\M \mapsto I_{\beta}(\M)$, i.e., the number of independent subsets of size $\beta$. Since all the corresponding Schubert matroids attain the maximum possible $\binom{n}{\beta}$, it follows that the number of independent sets of $\M$ of size $\beta$ is $\binom{n}{\beta}$ as well. In other words, $\alpha=\girth(\M) \geq \beta$.
\end{proof}

\begin{remark}\label{remark:girth-of-schubert}
    It is straightforward to check that the girth of a Schubert matroid, when seen as a lattice path matroid, is determined by the number of initial vertical segments of the upper path: if the upper path for $\mathsf{S}$ starts with $m$ vertical steps, then $\girth(\mathsf{S}) = m+1$. For example, the girths of the Schubert matroids depicted in Figure~\ref{fig:schuberts-4-2} are from left to right respectively $1$, $1$, $1$, $2$, $2$, and $3$. 
\end{remark}

\subsection{Simple and disconnected matroids}

A matroid being connected or simple is captured by the facial structure of our polytopes in a straightforward way.

\begin{proposition}\label{prop:disconnected}
    For every $n>1$, the class of disconnected matroids is facial in $\Omega_{r,n}$ and $\overline{\Omega}_{r,n}$.
\end{proposition}

\begin{proof}
    Consider the map $\M\mapsto \beta(\M)$ where $\beta$ denotes Crapo's beta invariant. This is a valuative invariant as it is a coefficient of the Tutte polynomial (see \cite[Corollary~5.7]{ardila-fink-rincon}). It is a well known fact that $\beta(\M) \geq 0$ for all $\M$, and that $\beta(\M) = 0$ if and only if $\M$ is disconnected or $\M\cong \U_{0,1}$. In particular, for $n>1$, $\beta$ attains its minimum exactly at disconnected matroids, which thus correspond to faces of our polytopes.
\end{proof}

\begin{remark}
    The minimal face of $\Omega_{2,4}$ containing all points coming from connected matroids must of course contain the points associated to $\mathsf{T}_{2,4}$ and $\U_{2,4}$, and in particular the full edge that joins them. By Example~\ref{ex:T24-non-extremal}, it follows that $\U_{1,2}\oplus \U_{1,2}$ belongs to this face. This says that \emph{connected} matroids are not facial in $\Omega_{2,4}$.
    Repeating the example with an appropriate number of loop and coloop summands included
    shows that connected matroids are not facial in $\Omega_{r,n}$ for any $2\le r\le n-2$. 
    
    Similarly, connected matroids are not facial in $\overline{\Omega}_{r,n}$. To prove this at $n = 4$ and $r=2$, take the matroids $\M_1$, $\M_2$, $\M_3$ with sets of bases
    \begin{align*} 
    \mathscr{B}(\M_1) &= \binom{[4]}{2} \smallsetminus \{\{1,2\}\},\\
    \mathscr{B}(\M_2) &= \binom{[4]}{2} \smallsetminus \{\{3,4\}\},\\
    \mathscr{B}(\M_3) &= \binom{[4]}{2} \smallsetminus \{\{1,2\},\{3,4\}\}.
    \end{align*}
We have that $\M_1,\M_2\cong \mathsf{T}_{2,4}$ (cf.\ Example~\ref{ex:T24-non-extremal}) are connected, and $\M_3 \cong \U_{1,2}\oplus \U_{1,2}$ is disconnected. Notice that there is a valuative relation coming from the split of a hypersimplex into two square pyramids:
    \[ \indicator{\mathscr{P}(\M_3)}  = -\indicator{\mathscr{P}(\U_{2,4})} + \indicator{\mathscr{P}(\M_1)} + \indicator{\mathscr{P}(\M_2)}.\]
The smallest face containing all points coming from connected matroids in $\mathcal{M}_{2,4}$ must in fact contain the three matroids appearing on the right hand side in the last display, and therefore also the disconnected matroid on the left.
\end{remark}

That loopless matroids and simple matroids are facial is the $m=2$ and $m=3$ cases of the next proposition.

\begin{proposition}\label{prop:simple-face}
    Let $m$ be a fixed non-negative integer. The class of all matroids whose girth is at least $m$ is facial in $\Omega_{r,n}$ and $\overline{\Omega}_{r,n}$. Moreover,
    \begin{align*}
        \dim(\text{matroids in $\Omega_{r,n}$ with $\girth(\M)\geq m$}) &= \binom{n-m+1}{r-m+1} - 1 \\
        \dim(\text{matroids in $\overline{\Omega}_{r,n}$ with $\girth(\M)\geq m$}) & =  [x^{r-m+1}] \mathrm{H}_{\U_{n-m+1,n}}(x) - 1.
    \end{align*}
    In the second equation, $\mathrm{H}_{\U_{r,n}}(x)$ stands for the augmented Chow polynomial of the uniform matroid $\U_{r,n}$. In particular,
    \begin{align*}
        \dim(\text{simple matroids in $\Omega_{r,n}$}) &= \binom{n-2}{r-2} - 1 \\
        \dim(\text{simple matroids in $\overline{\Omega}_{r,n}$}) & =  [x^{r-2}] \mathrm{H}_{\U_{n-2,n}}(x) - 1.
    \end{align*}
\end{proposition}

\begin{proof}
    Consider the map $\M \mapsto I_{m-1}$, where $I_{m-1}$ stands for the number of independent sets of rank $m-1$ of $\M$. Notice that  $|I_{m-1}| = [x^{r-m}]T_{\M}(x+1,1)$ where $T_{\M}(x,y)$ is the Tutte polynomial of $\M$; the map is therefore a valuation by \cite[Corollary~5.7]{ardila-fink-rincon}. This map attains its maximum exactly at the matroids for which all the sets of size $m-1$ are independent, i.e., at the matroids whose girth is at least $m$. In particular, this class of matroids must form a face in each of the two polytopes.

    In light of Lemma~\ref{lemma:girth}, the non-zero coordinates of a  matroid of girth $\ge m$ (resp.\ its isomorphism class) in $\overline{\Omega}_{r,n}$  (resp.\ $\Omega_{r,n}$) correspond to labelled (resp.\ isomorphism classes of) Schubert matroids of girth $\ge m$. 
    So to prove the equations for dimension in the proposition,
    it is enough to count these Schubert matroids (cf.\ Theorem~\ref{thm:dimensions-of-polytopes}) and show:
    \begin{align*}
        \#(\text{matroids in $\mathcal{S}_{r,n}$ with $\girth(\M)\geq m$}) &= \binom{n-m+1}{r-m+1} - 1 \\
        \#(\text{classes in $\overline{\mathcal{S}}_{r,n}$ with $\girth(\M)\geq m$}) & =  [x^{r-m+1}] \mathrm{H}_{\U_{n-m+1,n}}(x) - 1.
    \end{align*}
    The first equation follows from Remark~\ref{remark:girth-of-schubert}, as one must enumerate lattice paths from $(0,0)$ to $(n-r,r)$ that start with (at least) $m-1$ upper steps.  The second follows from the main theorem proved by Hoster in \cite{hoster}, which enables the enumeration of (labelled) Schubert matroids with given (co)girth via augmented Chow polynomials of uniform matroids.
\end{proof}

Chow polynomials of uniform matroids were described in \cite[Theorem~1.9]{ferroni-matherne-stevens-vecchi}, where a closed expression for $[x^i]\mathrm{H}_{\U_{r,n}}(x)$ in terms of Eulerian numbers is recorded, making it possible to compute these dimensions fairly explicitly. As an alternative, one can use the combinatorial formula found by Hameister, Rao, and Simpson in \cite[Theorem~5.1]{hameister-rao-simpson}. Notice that the case $m=2$ in the above statement, together with the fact that the augmented Chow polynomial of $\U_{n-1,n}$ is an Eulerian polynomial, gives that the dimension of the loopless face in $\overline{\Omega}_{r,n}$ is $A_{n-1,r-1}-1$, where $A_{n-1,r-1}$ is an Eulerian number.

\begin{remark}
    In Section~6.2 of his paper \cite{hampe} introducing the \emph{intersection ring of matroids},
    Hampe left the following open problem: study the polytope that arises as the convex hull of all matroid classes in that ring. 
    Only loopless matroids give classes, and these expand into loopless Schubert matroid classes with the same coefficients as in Theorem~\ref{thm:indicator-functions-schubert};
    so in the language of the present paper Hampe asked for the convex hull of all $\overline{p}_{\M}$ for $\M\in \overline{\mathcal{M}}_{r,n}$ without loops.
    Let us temporarily call this $\underline{\overline{\Omega}}_{r,n}$
    (with the underline riffing on the notational usage of \cite{semismall}). 
    
    Proposition~\ref{prop:simple-face} with $m=2$ implies that $\underline{\overline{\Omega}}_{r,n}$ is a face of $\overline{\Omega}_{r,n}$.
    Moreover, because the valuative relations for $\overline{\mathcal{M}}_{r,n}$ are generated by 
    relations in which every matroid appearing has an identical set of loops (Remark~\ref{rem:definitions-of-valuation})---we could say that matroids with different sets of loops cannot interact valuatively---we have that $\overline{\Omega}_{r,n}$ is the join of the family consisting of $\binom{n}{m}$ copies of~$\underline{\overline{\Omega}}_{r,m}$ for each $r\le m\le n$,
    each embedded in the coordinate subspace of Schubert matroids with a fixed set of $n-m$ loops.
\end{remark}

\subsection{Split matroids and their subclasses}

A matroid $\M$ of rank $r$ is said to be \emph{paving} if all of its subsets of cardinality $r-1$ are independent, i.e., if $\girth(\M)\geq r$. If a paving matroid $\M$ has a dual which is also paving, then one says that $\M$ is a \emph{sparse paving matroid}. A famous conjecture, originating from a hypothesis by Crapo and Rota in \cite[p.~317]{crapo-rota}, asserts that asymptotically almost all matroids are sparse paving. We refer to \cite{mayhew-etal} for more precise conjectures. The main result of \cite{pendavingh-vanderpol} provides partial evidence for these beliefs. 

As we show below, these classes of matroids appear as faces of our polytopes.

\begin{proposition}\label{prop:paving-copaving-faces}
    The classes of paving matroids and copaving matroids are facial in $\Omega_{r,n}$ and $\overline{\Omega}_{r,n}$. Moreover,
    \begin{align*}
        \dim(\text{paving matroids in $\Omega_{r,n}$}) &= n-r, \\
        \dim(\text{paving matroids in $\overline{\Omega}_{r,n}$}) &= \sum_{j=r}^{n} \binom{n}{j} - 1,
    \end{align*}
    and symmetrically for copaving matroids.
\end{proposition}

\begin{proof}
    Reasoning as in the proof of Proposition~\ref{prop:simple-face}, the map $\M\mapsto I_{r-1}$, where $I_{r-1}$ denotes the number of  independent sets of rank $r-1$ in $\M$, is a valuation attaining its the maximum value $\binom{n}{r-1}$ exactly at paving matroids. Thus, paving matroids correspond to faces of $\Omega_{r,n}$ and $\overline{\Omega}_{r,n}$. 
    
    From Lemma~\ref{lemma:girth}, we have that a matroid is paving if and only if all the Schubert matroids appearing on the right-hand-side of equation~\eqref{eq:integer-combination-of-Schuberts} with non-zero coefficients are paving and Schubert. Thus, to compute the dimension of the ``paving face'' of $\Omega_{r,n}$ and $\overline{\Omega}_{r,n}$, we shall enumerate paving Schubert matroids. The class of paving Schubert matroids is closely related to the ``panhandle matroids'' studied in \cite{hanely-et-al}. For every $1\leq r\leq n$ we have:
    \begin{align*}
        \dim(\text{paving matroids in $\Omega_{r,n}$}) &= \#\left\{[\mathsf{S}]\in \mathcal{S}_{r,n} : \mathsf{S} \text{ paving}\right\} - 1 = n-r, \\
        \dim(\text{paving matroids in $\overline{\Omega}_{r,n}$}) & = \#\left\{\mathsf{S}\in \overline{\mathcal{S}}_{r,n} : \mathsf{S} \text{ paving}\right\} - 1= \sum_{j=r}^{n} \binom{n}{j} - 1.
    \end{align*}
    The first equation is a consequence of Remark~\ref{remark:girth-of-schubert} because isomorphism classes of paving Schubert matroids correspond to lattice paths starting with $r-1$ vertical steps, leaving $n-r+1$ possibilities of where the $r$\/th vertical step is taken. The second equality follows from the fact that the isomorphism class corresponding to the path starting with $r-1$ upward steps and having the $r$\/th upward step in position $j$ can be labelled in $\binom{n}{j}$ different ways.
\end{proof}

\begin{proposition}
    The class of sparse paving matroids is facial in $\Omega_{r,n}$ and $\overline{\Omega}_{r,n}$. Moreover,
    \begin{align*}
        \dim(\text{sparse paving matroids in $\Omega_{r,n}$}) &= 1, \\
        \dim(\text{sparse paving matroids in $\overline{\Omega}_{r,n}$}) &= \binom{n}{r}.
    \end{align*}
\end{proposition}

\begin{proof}
    In both polytopes, the intersection of the faces coming from paving and copaving matroids yields the desired sparse paving face. For the computation of the dimensions we need to enumerate sparse paving Schubert matroids. There are only two isomorphism classes: the uniform matroid, and the matroid obtained from the uniform matroid by removing exactly one basis. There are $\binom{n}{r}+1$ (labelled) sparse paving Schubert matroids: the isomorphism class of the uniform matroid admits only one labelling, whereas the other one admits $\binom{n}{r}$ (which is determined by the only non-basis). 
\end{proof}

Given that sparse paving matroids form an edge in $\Omega_{r,n}$, one of whose vertices is the uniform matroid, it is natural to ask for the other vertex. 

\begin{theorem}\label{thm:extremal-sparse-paving}
    Let $\M\in \mathcal{M}_{r,n}$ be a sparse paving matroid. The following two assertions are equivalent:
    \begin{enumerate}[\normalfont (i)]
        \item $\M$ is an extremal matroid.
        \item $\M$ is a uniform matroid, or $\M$ maximizes the number of non-bases.
    \end{enumerate}
\end{theorem}

\begin{proof}
    (i) $\Rightarrow$ (ii). Assume that $\M$ is sparse paving and extremal. Notice that for a sparse paving matroid circuit-hyperplanes and non-bases coincide. Assume that $f:\mathcal{M}_{r,n}\to \mathbb{R}$ is a valuative invariant that attains its maximum at $\M$ and no other point of $\Omega_{r,n}$. Using \cite[Corollary~5.4]{ferroni-schroter}, we compute that
        \[ f(\M) = f(\U_{r,n}) - \lambda( f(\mathsf{T}_{r,n}) - f(\U_{r-1,r}\oplus \U_{1,n-r})).\]
    where $\lambda$ denotes the number of non-bases of $\M$, and $\mathsf{T}_{r,n}$ is the minimal matroid. 
    Let us denote $\alpha:= f(\U_{r,n})$ and $\beta:= f(\mathsf{T}_{r,n}) - f(\U_{r-1,r}\oplus \U_{1,n-r})$. Since $f$ attains its maximum at $\M$, if $\N$ is another sparse paving matroid we must have the inequality
    \[\alpha + |\{\text{non-bases of $\N$}\}|\cdot\beta \leq \alpha + |\{\text{non-bases of $\M$}\}|\cdot\beta\] 
    Depending on whether $\beta< 0$ or $\beta> 0$, this says that either $\lambda=0$ and hence $\M$ is the uniform matroid, or $\lambda$ is precisely the maximum possible number of non-bases for sparse paving matroids of this rank and size.
    (The case $\beta=0$ is excluded as it fails to select a single vertex of $\Omega_{r,n}$.)
    
    \noindent (ii) $\Rightarrow$ (i). We will use Lemma~\ref{lemma:extremal-sequence-invariants}. We already know that uniform matroids are extremal, so let us assume that $\M$ is a sparse paving matroid with the maximal possible number of non-bases. Consider the following two valuative invariants:
        \begin{align*}
            f(\M) &= [x^1]T_{\M}(x+1,1) + [y^1]T_{\M}(1,y+1),\\
            g(\M) &= \binom{n}{r} - T_{\M}(1,1).
        \end{align*}
    The first enumerates the independent sets of rank $r-1$ in $\M$ and the independent sets of rank $n-r-1$ in $\M^*$. This invariant is always at most $\binom{n}{r-1} + \binom{n}{n-r-1}$, and equality is attained only for sparse paving matroids. The second invariant counts the number of non-bases of $\M$. The matroids described in (ii) are exactly the ones that maximize $g$ among all the matroids that maximize $f$. Also, from \cite[Corollary~5.4]{ferroni-schroter} it follows that all these matroids yield the same point $p\in \mathbb{R}^{\mathcal{S}_{r,n}}$.
\end{proof}

The preceding result can be used to show that other famous families of matroids \emph{are not} facial.

\begin{corollary}
    The family of transversal matroids is not facial in $\Omega_{r,n}$. Similarly, the family of (matroids that under some ground set ordering become) positroids is not facial in $\Omega_{r,n}$. 
\end{corollary}

\begin{proof}
    To see this, notice that the matroid obtained by removing only one basis from $\U_{r,n}$ is transversal and isomorphic to a positroid. Since $\U_{r,n}$ itself is a transversal matroid and a positroid, if either of these two classes spanned a face of $\Omega_{r,n}$, then actually all extremal sparse paving matroids should be transversal matroids or isomorphic to positroids. One can check this does not happen for transversal matroids when $(r,n)=(2,6)$, or for positroids when $(r,n)=(3,6)$.
\end{proof}

\begin{remark}
    A longstanding conjecture asserts that sparse paving matroids ``predominate'' among matroids; see \cite{mayhew-etal} for a more precise statement. If this conjecture is true, then the number of vertices of the polytope $\Omega_{r,n}$ is negligible compared to the size of $\mathcal{M}_{r,n}$, because there are only two vertices populated by sparse paving matroids. A subtle caveat is that this is not the same as asserting that the number of isomorphism classes of extremal matroids is negligible when compared to the size of $\mathcal{M}_{r,n}$: proving this assertion would require a careful analysis of how many distinct isomorphism classes of sparse paving matroids sit at the same point of $\Omega_{r,n}$. We refer to Table~\ref{table:number-matroids} for a comparison between the number of vertices of $\Omega_{r,n}$ and the number of isomorphism classes of extremal matroids in $\mathcal{M}_{r,n}$.
\end{remark}

A generalization of paving and sparse paving matroids arises when considering compatible splits of a hypersimplex, as in the work of Joswig and Schr\"oter \cite{joswig-schroter}. In that paper they defined the class of \emph{split matroids}, which were later slightly modified in \cite{berczi} and \cite{ferroni-schroter}, giving rise to what are known as \emph{elementary split matroids}. The classes of split and elementary split matroids are essentially equal; they only differ on disconnected matroids. A matroid is said to be \emph{elementary split} if its nontrivial cyclic flats are pairwise incomparable.  Equivalently, these are the matroids that do not contain $\U_{0,1}\oplus\U_{1,1}\oplus \U_{1,2}$ as a minor \cite[Theorem~11]{berczi}.
A further characterization in terms of relaxations is found in \cite[Theorem~4.8]{ferroni-schroter}.

The following is a reformulation of one of the main results of Ferroni and Schr\"oter \cite[Theorem~5.3]{ferroni-schroter}.

\begin{theorem}
    Let $\M\in\overline{\mathcal{M}}_{r,n}$ be a matroid. Denote by $\overline{\mathcal{S}}(\M)$ the set of all the Schubert matroids $\mathsf{S}\in\overline{\mathcal{S}}_{r,n}$ for which the coefficients in equation~\eqref{eq:integer-combination-of-Schuberts} are non-zero. Then $\M$ is elementary split if and only if all the matroids in $\overline{\mathcal{S}}(\M)$ are elementary split.
\end{theorem}

Using the preceding theorem similarly to how we have used Lemma~\ref{lemma:girth} above, we can deduce that elementary split matroids also form a face for our polytopes.

\begin{proposition}\label{prop:split-face}
    The class of elementary split matroids is facial in $\Omega_{r,n}$ and $\overline{\Omega}_{r,n}$. Moreover,
    \begin{align*}
        \dim(\text{elementary split matroids in $\Omega_{r,n}$}) &= r(n-r), \\
        \dim(\text{elementary split matroids in $\overline{\Omega}_{r,n}$}) &= \binom{n}{r} + \sum_{i=1}^r \sum_{j=1}^{n-r} \binom{n}{i+j} -1.
    \end{align*}
\end{proposition}

\begin{proof}
    By the above discussion, if $\M$ is elementary split, then all the Schubert matroids $\mathsf{S}$ appearing on the right~hand side of equation~\eqref{eq:integer-combination-of-Schuberts} are Schubert and elementary split. (The arXiv version of \cite{ferroni-schroter} names this class the ``cuspidal matroids''.)
    A computation in \cite[Proposition~4.1]{ferroni-schroter-tutte} shows that there are exactly $r(n-r)+1$ such matroids, whence the first equation in this statement follows. The idea is to consider Schubert elementary split  matroids as lattice path matroids: see for example Figure~\ref{fig:cuspidal}. The isomorphism class is determined by the location of the green point. It can be located at any point $(i,j)$ where $1\leq i\leq n-r$ and $1\leq j\leq r$ or, alternatively, on the top border of the rectangle, in which case its exact location is irrelevant and the matroid is  uniform. Hence, there are $r(n-r)+1$ non-isomorphic Schubert elementary split matroids. To enumerate the labelled Schubert elementary split matroids, proceed as follows: once fixed the position $(i,j)$ there are exactly $\binom{n}{i+j}$ ways of choosing the $i+j$ elements in the unique non-trivial cyclic flat of the Schubert elementary split matroid. We leave the remaining details  to the interested reader.
\end{proof}

    \begin{figure}[ht]
        \centering
        \begin{tikzpicture}[scale=0.40, line width=.4pt]
        
        \draw[line width=1.5pt,line cap=round] (0,0)--(0,3) -- (5,3)--(5,7)--(9,7);
        \draw (0,0) grid (9,7);
        \draw[decoration={brace,raise=7pt},decorate]
            (0,0) -- node[left=7pt] {$r-s$} (0,3);
        \draw (5,0) grid (9,7);
        \draw[decoration={brace,mirror, raise=4pt},decorate]
        (0,0) -- node[below=7pt] {$n-r$} (9,0);
        
        \draw[decoration={brace, raise=5pt},decorate]
        (5,7) -- node[above=7pt] {$h-s$} (9,7); 
        
        \draw[decoration={brace, raise=5pt},decorate]
        (9,7) -- node[right=7pt] {$r$} (9,0); 
        \node[circle,fill=Green,scale=0.5] at (5,3) {};
        \end{tikzpicture}
        \caption{A Schubert elementary split matroid.}
        \label{fig:cuspidal}
    \end{figure}

\subsection{Less well-behaved classes}

A matroid is said to be \emph{modular} whenever $\rk(F_1) + \rk(F_2) = \rk(F_1\wedge F_2) + \rk(F_1\vee F_2)$ for every pair of flats $F_1,F_2\in\mathcal{L}(\M)$. 

\begin{proposition}
    The class of modular matroids is facial in $\Omega_{r,n}$ and $\overline{\Omega}_{r,n}$.
\end{proposition}

\begin{proof}
    Consider the map $\M \mapsto W_{r-1} - W_1$ where $W_i$ denotes the number of rank~$i$ flats of~$\M$. By \cite[Corollary~8.4]{ferroni-schroter}, $W_{r-1}$ and $W_1$ are valuations. By Greene's hyperplane theorem \cite{greene}, our map always takes non-negative values, but attains the value zero only at modular matroids, which therefore must form a face of each of our two polytopes. 
\end{proof}

In contrast to the foregoing sections, it is not easy to describe the dimension of the faces that modular matroids span. Notice that it is not true that the non-zero coordinates of a modular matroid in the Schubert basis must necessarily correspond to modular Schubert matroids.

\begin{proposition}
    The class of series-parallel matroids corresponds to a hyperplane \emph{slice} of the polytopes $\Omega_{r,n}$ and $\overline{\Omega}_{r,n}$.
\end{proposition}

In other words, there is an affine hyperplane in the ambient space of $\overline{\Omega}_{r,n}$, resp.\ $\Omega_{r,n}$,
passing through the point $\overline{p}_{\M}$, resp.\ $p_{[\M]}$,
exactly when $\M$ is series-parallel.

\begin{proof}
    A classical result due to Brylawski \cite{brylawski} characterizes series-parallel matroids on a ground set of size at least~$2$ precisely as those matroids whose $\beta$-invariant is $1$. Since the $\beta$-invariant is a valuation, the result follows.
\end{proof}

We observe that the slicing hyperplane in the above proposition does not support a face but is parallel to the hyperplane that supports the face of disconnected matroids, Proposition~\ref{prop:disconnected}.

\section{Representability}\label{sec:extremal-non-representable}

A representation of a matroid brings with it many tools for proving non-negativity of invariants of that matroid,
especially building an algebraic variety from the representation and then applying positivity theorems from algebraic geometry. 
It is often very hard to generalize a proof of this origin to arbitrary matroids. 
However, since in most cases the invariant under study is valuative, it is natural to ask whether 
a valuative invariant taking non-negative values at all representable matroids must take non-negative values in general.
If so, this would circumvent the need to carry out the above generalization. 
The next theorem shows that this is not the case.

\begin{theorem}\label{thm:extremal-non-representable}
    There exist valuative invariants $f:\mathcal{M}_{r,n}\to \mathbb{R}$ that attain non-negative values at all matroids that are representable over some field, yet do attain a negative value at a non-representable matroid.
\end{theorem}

\begin{proof}
    It suffices to show that for adequate choices of $n$ and~$r$, one may find a vertex $p$ of $\Omega_{r,n}$ having the property that all the isomorphism classes $[\M]\in \mathcal{M}_{r,n}$ for which $p_{[\M]} = p$ are non-representable.
    
    Let us consider simple matroids of rank~$3$ on $19$ elements with the maximum number of flats of rank~$2$ having size~$7$. Among those, consider the ones that attain the maximum number of flats of rank~$2$ having size~$3$. Let us denote by $\mathcal{C}\subseteq \mathcal{M}_{3,19}$ the resulting set of isomorphism classes. We will divide the proof into three steps. 
    \begin{enumerate}[Step (i)]
        \item We will characterize all the elements in $\mathcal{C}$.
        \item Using the characterization, we will show that all the matroids in $\mathcal{C}$ sit at the same point $p$. This will allow us to apply Lemma~\ref{lemma:extremal-sequence-invariants} and conclude that the point $p$ is a vertex.
        \item We will finally show that none of the elements in $\mathcal{C}$ is representable. 
    \end{enumerate}

    \noindent\emph{Step (i).} For a simple matroid $\M\in \mathcal{M}_{3,19}$ there are at most three flats of rank $2$ and size $7$. 
    To see this, think of the geometric representation of $\M$ as a configuration of points and (quasi)lines. Having four $7$-point lines is not possible, because being a simple matroid requires each pair of lines to intersect in at most one point, and $7\cdot 4 - \binom{4}{2} = 22 > 19$. One may easily produce a configuration with exactly three $7$-point lines.
    
    Provided that we have have exactly $3$ lines containing exactly $7$ points each, the maximum possible number of $3$-point lines is $36$. We first show this is an upper bound.
    Being simple guarantees that each pair of points of~$\M$ spans a rank $2$ flat (i.e., a line). However, if $\ell_i$ denotes the number of $i$-point lines, we have the equality
        \[ \sum_{i} \ell_i \binom{i}{2} = \binom{19}{2}.\]
    In particular,
        \[\ell_3\binom{3}{2} + \ell_7\binom{7}{2} \leq \binom{19}{2}.\]
    In the case of our interest, we have $\ell_7 = 3$, so that $\ell_3 \leq\frac{1}{3}\left(\binom{19}{2} - 3\binom{7}{2}\right) = 36$.
    
    By the reasoning above, in a configuration attaining $\ell_3=36$ (should one exist),
    we must have $\ell_i = 0$ for all $i\neq 3,7$. That is, every pair of points spans one of our $7$-point lines or $3$-point lines.  

    We now show that in any such potential configuration there cannot be a point $x$ lying on exactly two of the three $7$-point lines. Indeed, there are $12$ points on these two $7$-point lines different from $x$. Since every line contains $3$ points or $7$ points, the remaining $19-1-12=6$ points of the configuration must fall two each onto three $3$-point lines passing through $x$. However, in such a scenario there is no room for the third $7$-point line. 
    
    As there is also no room for three disjoint $7$-point lines on $19$ elements,
    the preceding contradiction tells us that in any of our potential matroids, the three $7$-point lines are concurrent. To see that $\ell_3=36$ is attainable in this case, notice that our problem is that of making a Latin square of order $6$, i.e., finding a $6\times 6$ matrix whose entries (sometimes called ``letters'') are integers from $1$ to $6$, in such a way that each letter appears exactly once in each row and each column. Think of the points on the $7$-point lines aside from the point of concurrency as indexing rows, columns, and letters.

    The geometric representation (as \cite{oxley} calls it) of one such matroid is shown in Figure~\ref{fig:extremal-19-3},
    with some of the $3$-point lines faded to gray to relieve the visual clutter.

    \begin{figure}[htb]
    \includegraphics[scale=0.95]{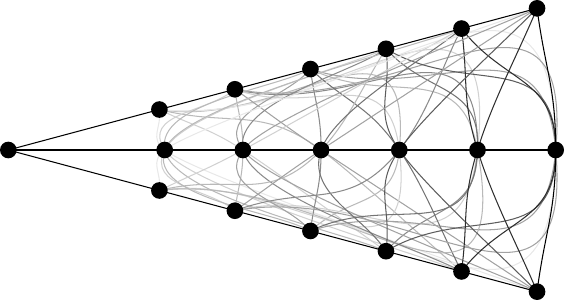}
    \caption{A nonrepresentable extremal rank~$3$ matroid on $19$ elements.}
    \label{fig:extremal-19-3}
    \end{figure}

    \noindent\emph{Step (ii).} Step~(i) shows that the matroids in our class $\mathcal{C}\subseteq \mathcal{M}_{19,3}$ are simple, and all of them have exactly $39$ rank~$2$ flats, $3$ of which have size~$7$, and $36$ of which have size~$3$. With this we can prove that all of them yield the same point $p\in \Omega_{3,19}$. Indeed, a simple matroid has girth is at least $3$, which in rank~$3$ implies that the matroid is paving, so that \cite[Theorem~5.3]{ferroni-schroter} applied to paving matroids implies that all these matroids sit at the same point. Now we can apply Lemma~\ref{lemma:extremal-sequence-invariants}, because the number of rank~$r$ flats of size~$s$ is a valuative invariant. We conclude that all of our matroids in $\mathcal{C}$ are extremal.
    
    \noindent\emph{Step (iii).} To see why none of the matroids in $\mathcal{C}$ is representable over any field $\mathbb{F}$, we proceed by contradiction. By the characterization above, any such matroid $\M$ has three $7$-point lines, say $L_0$, $L_1$, and $L_2$, that are concurrent at a point $x$. Let us take coordinates so that $x$ is a point at infinity, say the ``horizontal'' one. The three lines $L_0$, $L_1$, $L_2$ through $x$, written in affine coordinates, are of the form $L_i=\{(t,y_i) : t \in \mathbb{F}\}$ for some scalars $y_0,y_1,y_2\in \mathbb{F}$.  Given points $a = (t_a,y_0)$, $b = (t_b,y_0)$, $c = (t_c,y_0)$ on $L_0$, choose any $d$ on~$L_1$ and $f$ on~$L_2$ so that $adf$ is a line.  Let $bf$ meet $L_1$ in $e$, and $cd$ meet $L_2$ in $g$.  Then $eg$ meets $L_0$ in $b+c-a=(t_b+t_c-t_a, y_0)$.  This means that the $x$-coordinates of the finite points of~$\M$ lying on~$L_0$ form a coset of an additive subgroup of $\mathbb{F}$.  But there are six such points, and no field has an additive subgroup of order~$6$, contradiction.
\end{proof}

\begin{remark}
    Representability is preserved by adding or removing loops and coloops. Remark~\ref{remark:loops-coloops} also guarantees that extremality is preserved under adding or removing loops and coloops.  Therefore, the preceding construction extends, by adding suitable loops or coloops, to all pairs $(r,n)$ such that $r\geq 3$ and $n-r\geq 16$. 
\end{remark}

\begin{remark}
    Similarly to how Ingleton's inequality (see Oxley \cite[p.~169]{oxley}) provides a test for representability as a condition that the rank function must satisfy,
    one can interpret Theorem~\ref{thm:extremal-non-representable} as an assertion that there exists a valuative invariant which serves as a test for representability,
    at least in the cases $r\ge3$, $n-r\ge16$.
    It would be interesting to find more such valuative tests for representability.
\end{remark}

\section{Final remarks and conjectures}\label{sec:final-remarks}

\subsection{Extremal matroids and the pursuit of counterexamples in matroid theory}

Within matroid theory there remain many open problems asserting that some valuative invariant of matroids is non-negative.

These problems can be interpreted in our framework as a describing a linear functional on our polytopes $\Omega_{r,n}$. In particular, showing that the conjecture is true \text{is equivalent} to proving it for all extremal matroids. 

Let us comment briefly on two conjectures that have been disproved recently. 

\begin{itemize}
\item The conjecture by De Loera, Haws, and K\"oppe \cite{deloera-haws-koppe} on the positivity of the Ehrhart coefficients of a matroid polytope. This was disproved by Ferroni in \cite{ferroni3}, essentially by showing that the extremal matroids described in Theorem~\ref{thm:extremal-sparse-paving} attain a negative coefficient when $r$ and $n$ are sufficiently large. 

\item The Merino--Welsh conjecture for loopless and coloopless matroids \cite{merino-welsh}, which was settled negatively by Beke, Cs\'aji, Csikv\'ari, and Pituk \cite{beke-csaji-csikvari-pituk}. This conjecture can be proved to be equivalent\footnote{To be very precise, the conjecture asserted that for every matroid without loops and coloops $\max(T_{\M}(2,0),T_{\M}(0,2)) \geq T_{\M}(1,1)$. The validity of this inequality for all matroids without loops and coloops is equivalent to the validity of $T_{\M}(2,0)+T_{\M}(0,2)\geq 2\,T_{\M}(1,1)$ for all matroids without loops and coloops. Notice that, for a \emph{specific} matroid $\M$, the second inequality implies the first one, but perhaps not conversely.} to the Tutte polynomial inequality $T_{\M}(2,0)+T_{\M}(0,2) - 2\,T_{\M}(1,1)\geq 0$, which is a positivity conjecture for a valuative invariant. The extremal matroid violating this inequality is obtained for the choices $r=22$ and $n=33$ in Example~\ref{example:thickening-uniform}.
\end{itemize}
Although a characterization of all extremal matroids seems out of reach, we feel motivated to invite the reader to find more meaningful classes:

\begin{problem}
    Identify classes of matroids that are extremal.
\end{problem}

Specifically, Conjecture~\ref{conj:direct-sums} would provide an important such class.  
We also conjecture some particular answers to this problem, firstly:

\begin{conjecture}\label{conj:braid}
    For every $n\geq 1$, the graphic matroid $\mathsf{K}_n$ of the complete graph on $n$ vertices is an extremal matroid.
\end{conjecture}

One fact that leads us to believe in this conjecture is a theorem by Bonin and Miller \cite[Theorem~3.2]{bonin-miller}, which characterizes  graphic matroids of complete graphs (also known as \emph{braid matroids}) in terms of certain valuative invariants. Unfortunately, as the characterization is not in terms of a sequence of valuative invariants being \emph{maximized}, the extremality of $\mathsf{K}_n$ does not follow easily. However, we believe that some adjustments of the strategy can lead to a full proof of our conjecture. Similarly, we propose the following:

\begin{conjecture}\label{conj:proj-dowl}\mbox{}
\begin{enumerate}
\item    Projective geometries are extremal matroids. 
\item    Dowling geometries are extremal matroids. 
\end{enumerate}
\end{conjecture}

Results by Kung \cite{kung}, Bennet, Bogart, and Bonin \cite{bennet-bogart-bonin}, and the aforementioned paper by Bonin and Miller \cite{bonin-miller} seem to be promising points of departure toward this conjecture.\footnote{After this article was posted on the arXiv, Joseph Bonin \cite{bonin} announced a proof for Conjectures~\ref{conj:braid} and \ref{conj:proj-dowl}.}

\subsection{Extremality and uniqueness for the Tutte polynomial}

If $f:\overline{\mathcal{M}}_{r,n}\to V$ is a valuative function valued in a real vector space $V$,
then $\conv\{f(\M):\M\in\overline{\mathcal{M}}_{r,n}\}\subseteq V$, which we might call ``the polytope of~$f$'',
is a linear projection of $\overline{\Omega}_{r,n}$,
for the same reason as in Remark~\ref{rem:symmetrization-is-a-projection}.
If $f$ is moreover an invariant then the polytope of~$f$ is a projection of $\Omega_{r,n}$.

An example comes from the Tutte polynomial, surely the most studied polynomial matroid invariant, which is valuative \cite[Corollary~5.7]{ardila-fink-rincon}.
Work of Kung implies that the polytope of the Tutte polynomial has dimension $r(n-r)$:
\cite[Corollary 6.10]{kung-syzygies} shows that this polytope \emph{linearly} spans an $r(n-r)+1$ dimensional subspace of $\mathbb{R}[x,y]$ (see also \cite[Corollary~4.2]{ferroni-schroter-tutte})
but by \cite[Lemma~3.5]{kung-syzygies} all its vertices inhabit an affine hyperplane.

There is literature on \emph{Tutte unique} matroids, those determined up to isomorphism by their Tutte polynomial.
The two matroids $\M,\N$ of Remark~\ref{rem:nonisomorphic-same-vertex}
are better known as a nonexample thereof, i.e., for having equal Tutte polynomials:
from our point of view this is a consequence of the equality $p_{[\M]} = p_{[\N]}$.
Our work in this article suggests that it would be of interest also to study ``Tutte extremal'' matroids,
those determining vertices of the polytope of the Tutte polynomial.
This appears to be a basically independent problem:
in particular it seems that neither Tutte extremality nor our extremality implies the other, nor is either related by implication to valuative or Tutte uniqueness.

\subsection{Edges of the polytopes}

After our study of vertices of $\overline{\Omega}_{r,n}$ and $\Omega_{r,n}$, the faces of greater dimension warrant attention. The following question about the edges intrigues us.

\begin{question}
    When do two matroids $\M$ and $\N$ give a pair of neighboring vertices of $\overline{\Omega}_{r,n}$? When do two isomorphism classes of extremal matroids yield neighboring vertices of $\Omega_{r,n}$?
\end{question}

The second part, concerning edges of the polytope $\Omega_{r,n}$, seems to be much more difficult, as it is meaningful only for extremal matroids, of which we understand little.
A special case of the question arises when one of the two vertices is fixed to be the uniform matroid:

\begin{problem}
    Characterize or describe the matroids (resp.\ isomorphism classes of matroids) that yield a vertex neighboring the point $\overline{p}_{\U_{r,n}}$ in $\overline{\Omega}_{r,n}$ (resp. $p_{\U_{k,n}}$ in $\Omega_{r,n}$). 
\end{problem}

\bibliographystyle{amsalpha}
\bibliography{bibliography}

\end{document}